\documentclass[a4paper,12pt,draft]{amsart}
\usepackage[margin=30mm]{geometry}
\usepackage{amssymb, amsmath, amsthm, amscd}

\usepackage[T1]{fontenc}
\usepackage{eucal,mathrsfs,dsfont}%gives nicer set-names. Use \mathds instead of \mathds
\usepackage{color}%cyan,red;magenta,green;yellow,blue
\definecolor{mno}{rgb}{0.5,0.1,0.5}

\newcommand{\R}{\mathds R}

\newcommand{\Pp}{\mathds P}
\newcommand{\Ee}{\mathds E}

\newcommand{\I}{\mathds 1}

\newcommand{\Bb}{\mathscr{B}}

\def\<{\langle}
\def\>{\rangle}

\newcommand{\scalar}[1]{\langle#1\rangle}

\newtheorem{theorem}{Theorem}[section]
\newtheorem{lemma}[theorem]{Lemma}
\newtheorem{proposition}[theorem]{Proposition}
\newtheorem{corollary}[theorem]{Corollary}

\theoremstyle{definition}

\newtheorem{remark}[theorem]{Remark}

\begin{document}
\allowdisplaybreaks
\title[SDEs Driven by Multiplicative L\'{e}vy Noises]
{\bfseries Gradient Estimates and Ergodicity for SDEs Driven by
Multiplicative L\'{e}vy Noises via Coupling}

\author{Mingjie Liang \qquad Jian Wang}
\thanks{\emph{M.\ Liang:}
College of Mathematics and Informatics, Fujian Normal University,
350007 Fuzhou, P.R. China. \texttt{liangmingjie@aliyun.com}}
\thanks{\emph{J.\ Wang:}
College of Mathematics and Informatics \& Fujian Key Laboratory of Mathematical Analysis and Applications (FJKLMAA), Fujian Normal University, 350007 Fuzhou, P.R. China. \texttt{jianwang@fjnu.edu.cn}}
\date{}

\maketitle

\begin{abstract}  We consider SDEs driven by multiplicative pure jump L\'{e}vy noises, where L\'evy processes are not necessarily comparable
to $\alpha$-stable-like processes. By assuming that the SDE has a
unique solution, we obtain gradient estimates of the associated
semigroup when the drift term is locally H\"{o}lder continuous, and
we establish the ergodicity of the process both in the
$L^1$-Wasserstein distance and the total variation, when the
coefficients are dissipative for large distances. The proof is
based on a new explicit Markov coupling for SDEs driven by
multiplicative pure jump L\'{e}vy noises, which is derived for the first time in this paper.

\medskip

\medskip

\end{abstract}
\allowdisplaybreaks

\section{Introduction and main results}\label{section1}
We consider the following $d$-dimensional stochastic differential equation (SDE)
driven by multiplicative pure jump L\'{e}vy noises
\begin{equation}\label{s1} {d}X_t=b(X_{t})\,dt+\sigma(X_{t-})\,dZ_t,\quad  X_0=x\in \R^d,\end{equation}
where $b: \R^d\rightarrow\R^d$ is measurable, $\sigma:
\R^d\rightarrow\R^d\otimes\R^d$ is continuous, and
$Z:=(Z_t)_{t\ge0}$ is a pure jump L\'{e}vy process on $\R^d$, i.e.,
the finite-dimensional distributions of the process $Z$ are uniquely
characterized by the characteristic function $$\Ee
e^{i\scalar{\xi,Z_t}} = e^{-t\phi_Z(\xi)},\quad  \xi\in \R^d, t>0$$
with
  $$ \phi_Z(\xi) =\int\big(1-e^{i\langle{\xi},{z}\rangle}+i\langle{\xi},{z}\rangle\I_{\{|z|\le 1\}}(z)\big)\,\nu({d}z).$$
Here, $\nu$ is the L\'evy measure, i.e.,\ a $\sigma$-finite measure
on $(\R^d, \mathscr{B}(\R^d))$ such that $\nu(\{0\})=0$ and $\int(1\wedge |z|^2)\,\nu(dz)<\infty$.

Throughout this paper, we always assume that there exists a non-explosive and pathwise unique
solution to SDE \eqref{s1}, see \cite{BLP, BBC, HP,
Pr1, Pr2, Pr3,TTW, XWang,Z} for more details. We also need the
following two assumptions on the coefficient $\sigma(x)$:
\begin{itemize}
\item $\sigma(x)$ is uniformly non-degenerate in the sense that, there exists a constant $\Lambda\ge1$
such that for all $\xi\in \R^d$,
\begin{equation}\label{p:est0}\begin{split}\Lambda^{-1}|\xi|\le &\inf_{x\in \R^d}\{|\sigma(x)\xi| \wedge
|\sigma(x)^{-1}\xi|\} \le \sup_{x\in \R^d} \{|\sigma(x)\xi| \vee
|\sigma(x)^{-1}\xi|\} \le \Lambda|\xi|,\end{split}\end{equation}
where $|\cdot|$ denotes the Euclidean norm.

\item $\sigma(x)$ is bounded  and globally Lipschitz
continuous, i.e.,\ there is a constant $L_\sigma>0$ such that for
all $x,y\in \R^d$,
$$\|\sigma(x)-\sigma(y)\|_{{\rm H.S.}}\le L_\sigma|x-y|,$$ where $\|\cdot\|_{{\rm H.S.}}$ denotes the Hilbert-Schmidt norm of a matrix, and $L_\sigma$ is called the Lipschitz constant.
\end{itemize}

The  goal of the present paper is to establish the regularity of the
semigroups and the ergodicity of the process corresponding to the
SDE \eqref{s1} driven by  multiplicative pure jump L\'{e}vy noises.
More explicitly, we not only extend the main results of \cite{Lwcw}
to multiplicative L\'{e}vy noises setting, but also establish the
regularity of the semigroups when the drift term is locally
H\"{o}lder continuous. The L\'evy process in this paper can be
non-symmetric and not comparable with the $\alpha$-stable-type
process. The methods used in this paper rely on coupling techniques
and constructing coupling operators. We emphasize that the coupling
for SDEs driven by multiplicative pure jump L\'{e}vy noises, which
has been open for a long time, is derived for the first time in this
paper.

Coupling for SDEs driven by multiplicative Brownian motions is a
well developed field, and there is a vast literature on this topic;
we mention here the papers \cite{CL,LR,PW, Wang16} and the monographs \cite{Chen, Lin, Tho, Wangbook}.
Notice that, in contrast with the case of coupling for SDEs driven
by multiplicative Brownian motions, our case for multiplicative
L\'evy noises is quite different. Consider the SDE \eqref{s1} on $\R^d$
with the L\'evy process $Z$ replaced by a Brownian
motion $(B_t)_{t\ge0}$, i.e.,
\begin{equation}\label{s00} dX_t=b(X_t)\,dt+\sigma(X_t)\,dB_t,\quad X_0=x\in \R^d.\end{equation} Assume that for
any $x\in \R^d$, $\sigma(x)\sigma(x)^*\ge
\lambda_0\mathrm{I}_{d\times d}$ with some constant $\lambda_0>0$,
where $\sigma(x)^*$ is the transpose of $\sigma(x)$. We can
reformulate \eqref{s00} as
$$dX_t=b(X_t)\,dt+\sqrt{\lambda_0}\,dB_t'+\sigma_0(X_t)\,dB_t'',\quad X_0=x\in \R^d,$$
where $(B_t')_{t\ge0}$ and $(B_t'')_{t\ge0}$ are two independent
$d$-dimensional Brownian motions on $\R^d$, and $\sigma_0(x):\R^d\to
\R^d\otimes \R^d$ satisfies that
$\sigma(x)\sigma(x)^*=\lambda_0\mathrm{I}_{d\times
d}+\sigma_0(x)\sigma_0(x)^*$. By the formula above, one can reduce
the coupling for SDEs driven by multiplicative Brownian motions into
the case for additive Brownian motions. For example, we use coupling
by reflection for $(B'_t)_{t\ge0}$ and coupling by parallel
displacement for $(B_t'')_{t\ge0}$. Usually the term of coupling
by reflection for $(B'_t)_{t\ge0}$ plays a leading role in
applications, see \cite{PW, Wang16}. However, such nice additive
property fails if we apply to the SDE \eqref{s1}, i.e.,\ replace
$(B_t)_{t\ge0}$ by the L\'evy process $(Z_t)_{t\ge0}$ in the
argument above. We cannot  use the above technique based on the
decomposition, and in some sense the coupling for SDEs driven by
multiplicative L\'evy noises is highly non-trivial. Indeed, we will
construct the coupling for SDEs driven by multiplicative L\'evy
noises directly through the coupling operator for the associated
generator.

In the existing mathematical literature  there are a few works
devoted to coupling for SDEs with additive L\'{e}vy noises, i.e., the coefficient $\sigma(x)$ in \eqref{s1} is independent of the space
variable. The readers
can refer to \cite{Lwcw,M15,Wang} for an essential progress. In
particular, the couplings used in \cite{M15,Wang} depend heavily on
the existence of the rotationally symmetric component for L\'{e}vy
measure; while in the framework of \cite{Lwcw} only the existence of
absolutely continuous component of L\'{e}vy measure is required, and
then the main result of \cite{Lwcw} works for some non-symmetric and
even singular L\'{e}vy measure. However, there is no result for the
coupling for  SDEs driven by multiplicative L\'{e}vy processes till
now. The difficulty is due to the fact that in this situation an efficient
coupling shall pay attention to not only the shape of L\'evy measure
itself but also the diffusion coefficient, both of which are usually
hard to handle. An important contribution of this paper is to fill this gap.

\ \

To illustrate the contribution of our paper, we present the
following statement, which is a special case of our main results in
Section \ref{section44}. Denote  by $X:=(X_t)_{t\ge0}$ the unique
solution to the SDE \eqref{s1}. For any $f\in B_b(\R^d)$ (the set of
bounded measurable functions on $\R^d$), let
$$P_tf(x)=\Ee^xf(X_t),\quad x\in \R^d,t\ge0.$$ We will assume that one of assumptions below holds for the L\'evy measure $\nu$:
\begin{itemize}
\item[(i)] $$\nu(dz)\ge \I_{\{|z|\le \eta\}}\frac{c_{0}}{|z|^{d+\alpha}}\,dz$$ for some $\eta\in(0,1)$ and $c_{0}>0$.
\item[(ii)] When $\sigma(x)=(\sigma_{i,j}(x))_{d\times d}$ is diagonal, i.e.,\, $\sigma_{i,j}(x)=0$ for all $x\in \R^d$ and $1\le i\neq j\le d$,
 \begin{equation}\label{con:lm}\nu(dz)\ge \I_{\{0<z_1\le \eta\}}\frac{c_{0}}{|z|^{d+\alpha}}\,dz\end{equation}
for some $\eta\in(0,1)$ and $c_{0}>0$.
%\item[(iii)] When  $\sigma(x)=(\sigma_{i,j}(x))_{d\times d}$ is diagonal,
%$$\nu(dz)\ge \sum_{i=1}^d \frac{c_0}{|z_i|^{1+\alpha}}\I_{\{|z_i|\le \eta\}}\,dz_i$$for some $\eta\in(0,1)$ and $c_{0}>0$.
\end{itemize}

\begin{theorem}\label{T:thm1} Assume that the diffusion coefficient $\sigma(x)$ is bounded and Lipschitz continuous, and the drift term $b(x)$ is locally $\beta$-H\"{o}lder
continuous with $\beta\in((1-\alpha)\vee0,1]$ for some $\alpha\in
(0,2)$.
If one of assumptions $(i)$ and $(ii)$ above is satisfied for the L\'evy measure $\nu$, then the following hold.
\begin{itemize}
\item[(1)] If $\alpha\in (1,2)$, then for any $\theta>0$, there exists a constant $C_1:=C_1(\theta)>0$ such that for all $f\in B_b(\R^d)$ and $t>0$,
$$
  \sup_{x\neq y}\frac{{|P_t f(x)-P_t f(y)|}}{|x-y|}
  \le C_1\|f\|_\infty \left(\frac{\log^{1+\theta}(1/(t\wedge1))}{t\wedge1}\right)^{1/\alpha}.$$
  \item[(2)] If $\alpha\in(0,1]$, then, for any $\theta\in(0,\alpha)$, there exists a constant $C_2:=C_2(\theta)>0$ such that for all $f\in B_b(\R^d)$ and $t>0$, $$
  \sup_{x\neq y}\frac{{|P_t f(x)-P_t f(y)|}}{|x-y|^\theta}
  \le C_2\|f\|_\infty t^{-\theta/\alpha}.$$

 \end{itemize} \end{theorem}

Gradient estimates for the semigroup associated to SDEs driven by
multiplicative subordinated Brownian motions have been obtained in
\cite{WXZ} by using the Malliavin calculus and a finite-jump
approximation argument. For SDEs with multiplicative Brownian
motions and general Poisson jump processes, Takeuchi \cite{Tak}
obtained the derivative formula for the associated semigroups, by using stochastic
diffeomorphism flows and Girsanov's transformation. Later, based on
Bismut's approach to the Malliavin calculus with jumps, a derivative
formula of Bismut-Elworthy-Li's type was established in \cite{WXZ1}
for SDEs with multiplicative $\alpha$-stable-like processes and
maybe including non-degenerate diffusion part. It is clear that
under \eqref{con:lm}, the L\'evy process $(Z_t)_{t\ge0}$ is not
comparable with the $\alpha$-stable-like process, so that the tool based on the Malliavin calculus with
jumps used in \cite{WXZ,WXZ1} cannot apply.

  Next, we consider the ergodicity for the SDE given by \eqref{s1}. Let $\psi$ be a strictly increasing function on $[0,\infty)$ satisfying $\psi(0)=0$. Given two probability measures $\mu_1$ and $\mu_2$ on $\R^d$, define
  $$W_\psi(\mu_1,\mu_2)=\inf_{\Pi\in \mathscr{C}(\mu_1,\mu_2)}\int_{\R^d\times\R^d} \psi(|x-y|)\,d\Pi(x,y),$$
where $\mathscr{C}(\mu_1,\mu_2)$ is the collection of measures on
$\R^d\times\R^d$ having $\mu_1$ and $\mu_2$ as marginals. When
$\psi$ is concave, the above definition gives rise to a
{Wasserstein distance} $W_\psi$ in the space of probability
measures $\mu$ on $\R^d$ such that $\int \psi(|z|)\,\mu(d
z)<\infty$. If $\psi(r)=r$ for all $r\geq 0$, then $W_\psi$ is the
standard $L^1$-Wasserstein distance (with respect to the Euclidean
norm $|\cdot|$), which will be denoted by $W_1(\mu_1,\mu_2)$ for
simplicity. Another well-known example for $W_\psi$ is given by
$\psi(r)=\I_{(0,\infty)}(r)$, which leads to the total variation
distance $W_{\psi}(\mu_1,\mu_2)= \frac{1}{2} \|\mu_1-\mu_2\|_{{\rm
Var}}.$

 \begin{theorem} \label{T:222} Assume that the diffusion coefficient $\sigma(x)$ is  Lipschitz
continuous with Lipschiz constant $L_\sigma>0$.
Suppose furthermore that
\begin{itemize}
 \item[(i)] the L\'evy measure $\nu$ satisfies that \begin{equation}\label{e:ffgg}\int_{\{|z|\ge1\}}|z|\,\nu(dz)<\infty\end{equation} and one of assumptions $(i)$ and $(ii)$ before Theorem $\ref{T:thm1}$.
 \item[(ii)] the drift term $b(x)$ satisfies
\begin{equation}\label{e:tss1}\begin{split}
  \frac{\langle b(x)-b(y), x-y\rangle}{|x-y|}\le \begin{cases} K_1|x-y|^\beta,&\quad|x-y|<l_0,\\
  -K_2|x-y|,&\quad|x-y|\ge l_0\end{cases}\end{split}\end{equation}
 for all $x,y\in \R^d$ with some constants $\beta\in((1-\alpha)\vee0,1]$, $l_0\ge 0$ $ K_1\ge0$ and $K_2>0$. \end{itemize} Then, there exist constants $C,\lambda>0$ such that for all $x,y\in \R^d$ and $t>0$,
 $$W_1(\delta_xP_t, \delta_yP_t)\le C e^{-\lambda t} |x-y|$$ and \begin{equation}\label{e:varss} \|\delta_xP_t- \delta_yP_t\|_{{\rm Var}}\le C e^{-\lambda t}(1+|x-y|).\end{equation} Moreover, \eqref{e:varss} holds true even for $\beta=0$ in \eqref{e:tss1}.
 \end{theorem}

As a consequence of Theorem \ref{T:222}, we have
\begin{corollary}\label{C:c3}Under the setting of Theorem $\ref{T:222}$, there exist a unique probability measure $\mu$, some positive functions $C_1(x), C_2(x)$ and a constant $\lambda>0$ such that
$$  W_1(\delta_xP_t, \mu)\le C_1(x) e^{-\lambda t},\quad x\in \R^d, t>0,$$ and $$ \|\delta_xP_t- \mu\|_{{\rm Var}}\le C_2(x) e^{-\lambda t},\quad x\in \R^d,t>0.$$  \end{corollary}

When the coefficient $\sigma(x)$ is independent of the space
variable (i.e.,\ the noise is additive) and the drift term $b$ is
dissipative for large distances (i.e., satisfies \eqref{e:tss1}), in \cite{Wang} the second
author established the exponential convergence rate in the
$L^p$-Wasserstein distance for any $p\ge1$ when the L\'{e}vy noise
in \eqref{s1} has an $\alpha$-stable component.  For a large class
of L\'{e}vy processes whose associated L\'{e}vy measure has a
rotationally invariant absolutely continuous component, Majka
obtained in \cite{M15} the exponential convergence rates with
respect to both the $L^1$-Wasserstein distance and the total
variation. Recently, the results of \cite{M15,Wang} are extended and
improved in \cite{Lwcw}, where the associated L\'{e}vy measure of
L\'{e}vy process is only assumed to have an absolutely continuous
component. It is noticed that all the works above are restricted to
the additive noise case. Once the coefficient $\sigma(x)$ depends on
the space variable (i.e.,\ the noise is multiplicative), the problem gets
more complicated. When the coefficients are locally Lipschitz
continuous and satisfy a Lyapunov type dissipative condition, it has
been shown in \cite{A0,K0,M0} that there is a unique invariant
probability measure associated to the SDE \eqref{s1}, which is
exponentially ergodic. Recently, Xie and Zhang studied in
\cite{XWang} the exponential ergodicity of SDEs driven by general
multiplicative L\'evy noises (maybe with Brownian motions), when $b$
is locally bounded and maybe singular at infinity. To the best of
our knowledge, there is no result about the exponential convergence
rates with respect to the $L^1$-Wasserstein distance, when the
coefficient $\sigma(x)$ is bounded and Lipschitz continuous, and $b(x)$ is dissipative for large distances.

\ \

The remainder of this paper is arranged as follows. In Section \ref{section2},
we first review the refined basic coupling for L\'evy processes constructed in \cite{Lwcw},
 and then present a new coupling for multiplicative L\'{e}vy process, which is a key part of our paper.
 We also prove the existence of coupling process here.
 Section \ref{Section3} is devoted to some explicit estimates for the coupling operator, which is a necessary
 ingredient of our proof. General ideas to yield the regularity of the semigroups and the ergodicity of
 the process via coupling are presented in Section \ref{section44}.  Finally, we present proofs of Theorems \ref{T:thm1}, \ref{T:222} and Corollary \ref{C:c3} in the last section.

\section{Coupling operator and coupling process}\label{section2}
\subsection{Coupling operator for the SDE \eqref{s1}}
 Denote by $X:=(X_t)_{t\ge0}$ the solution to the SDE \eqref{s1}. It is easy to see that the generator $L$ of the process $X$ acting on $C_b^2(\R^d)$ is given by
  \begin{equation}\label{proofth21}
  \begin{split}
  Lf(x)=&\langle \nabla f(x),b(x)\rangle\\
  &+\int\!\!\Big(f(x+\sigma(x)z)-f(x)-\langle\nabla f(x), \sigma(x)z\rangle\I_{\{|z|\le 1\}}(z)\Big)\,\nu(\mathrm{d}z).
  \end{split}
  \end{equation} The purpose of this subsection is to construct a new and efficient coupling operator for $L$, which is one of crucial ingredients in our approach. Recall that an operator $\widetilde L$ acting on $C_b^2(\R^{2d})$ is a coupling of $L$, if for any $f,g\in C_b^2(\R^d)$,
  \begin{equation}\label{marginality}
  \widetilde L h(x,y)= L f(x)+ Lg(y),
  \end{equation} where $h(x,y)=f(x)+g(y)$ for all $x,y\in\R^d$.

 \subsubsection{{\bf Additive L\'evy noises}}
 To illustrate clearly ideas of the construction of a proper coupling for the SDE \eqref{s1}, in this part we briefly introduce
 \emph{the refined basic coupling operator} constructed in \cite[Section 2]{Lwcw} for SDEs driven by additive L\'evy noises
 (that is, the case that $\sigma(x)=\mathrm{I}_{d\times d}$ for any $x\in \R^d$ in the SDE \eqref{s1}).

Throughout this part, we consider the operator $L$ given by \eqref{proofth21}, where $\sigma(x)=\mathrm{I}_{d\times d}$ for all $x\in \R^d$.
Motivated by the (classical) \emph{basic coupling} for Markov $q$-processes or Makov chains
(see \cite[Example 2.10]{Chen} for instance), we can define the following basic coupling for the operator $L$. For any $h\in C_b^2(\R^{2d})$ and $x,y\in \R^d$,
  let \begin{align*}
  \widetilde L h(x,y)&= \<\nabla_x h(x,y), b(x)\>+\<\nabla_y h(x,y), b(y)\>\\
  &\quad+ \int \!\!\Big(h(x+z, y+z+(x-y))-h(x,y)-\<\nabla_x h(x,y), z\>\I_{\{|z|\leq 1\}}\\
  &\hskip30pt -\<\nabla_y h(x,y), z+(x-y)\>\I_{\{|z+(x-y)|\leq 1\}} \Big)\,\mu_{y-x}(dz)\\
  &\quad + \int \!\!\Big(h(x+z, y)-h(x,y)-\<\nabla_x h(x,y), z\>\I_{\{|z|\leq 1\}}\Big)\,(\nu-\mu_{y-x}) (dz)\\
  &\quad + \int \!\!\Big(h(x, y+z)-h(x,y)-\<\nabla_y h(x,y), z\>\I_{\{|z|\leq 1\}}\Big)\,(\nu-\mu_{x-y}) (dz).
  \end{align*}
Here, $\mu_{y-x}(dz):=[\nu\wedge (\delta_{y-x} \ast\nu)](dz)$, and $\nabla_xh(x,y)$ and $\nabla_yh(x,y)$ are defined as the gradient of $h(x,y)$ with respect to $x\in\R^d$ and $y\in \R^d$ respectively.
For the sake of our explanation, by the structure of the generator of L\'evy process, the coupling above can be simply written as follows:
  \begin{equation}\label{basic-coup-1}
  (x,y)\longrightarrow
    \begin{cases}
    (x+z, y+z+(x-y)), & \mu_{y-x}(dz);\\
    (x+z, y), & (\nu - \mu_{y-x})(dz);\\
    (x, y+z), & (\nu - \mu_{x-y})(dz).
    \end{cases}
  \end{equation} In the first row the distance between two marginals decreases from $|x-y|$ to $|(x+z)- (y+z+(x-y))|=0$, so this term plays a key role in coupling two marginals together. For this aim, its density enjoys the biggest jump rate $\mu_{y-x}(dz)$, i.e.,\, the maximum common part of the jump intensities from $x$ to $x+z$ and from $y$ to $y+z+(x-y)$. However, the second and the last rows are not so welcome (indeed they are not so easy to handle for L\'evy jumps), because the new distance are $|x-y+z|$ or $|x-y-z|$, which can be much bigger than the original distance $|x-y|$ when the jump size $z$ is large.

  To overcome this disadvantage, we figure out the following coupling:
  \begin{equation}\label{basic-coup-2}
  (x,y)\longrightarrow
    \begin{cases}
    (x+z, y+z+(x-y)), & \frac12 \mu_{y-x}(dz);\\
    (x+z, y+z+(y-x)), & \frac12 \mu_{x-y}(dz);\\
    (x+z, y+z), & \big(\nu - \frac12 \mu_{y-x}  -\frac12 \mu_{x-y}\big)(dz).
    \end{cases}
  \end{equation} Now, in the second row the distance after the jump is $2|x-y|$. Though it doubles the original distance, it is better than that in \eqref{basic-coup-1} when the jump size $z$ is large. (In some sense, it is also easier to deal with.) Besides, the distance remains unchanged in the last row.

  The above coupling \eqref{basic-coup-2} has a drawback too. For example, if the original pure jump L\'evy process is of finite range, then the jump intensity $\mu_{y-x}(dz)$ is identically zero when $|x-y|$ is large enough. Hence, two marginal processes of the coupling \eqref{basic-coup-2} will never get closer if they are initially far away. Therefore, we need further modify the coupling above.  Let $\kappa>0$, and for any $x$, $y\in\R^d$, define
$$
  (x-y)_{\kappa}=\bigg(1\wedge \frac{\kappa}{|x-y|}\bigg)(x-y).
$$ In \cite[Section 2]{Lwcw} we finally modify the coupling above into \begin{equation}\label{basic-coup-3}
  (x,y)\longrightarrow
    \begin{cases}
    (x+z, y+z+(x-y)_\kappa), & \frac12 \mu_{(y-x)_\kappa}(dz);\\
    (x+z, y+z+(y-x)_\kappa), & \frac12 \mu_{(x-y)_\kappa}(dz);\\
    (x+z, y+z), & \big(\nu - \frac12 \mu_{(y-x)_\kappa}  -\frac12 \mu_{(x-y)_\kappa}\big)(dz).
    \end{cases}
  \end{equation}
We see that if $|x-y|\leq \kappa$, then the above coupling is the same as that in \eqref{basic-coup-2}. If $|x-y|>\kappa$, then according to the first two rows, the distances after the jump are $|x-y|-\kappa$ and $|x-y|+\kappa$, respectively. Therefore, the parameter $\kappa$ serves as the threshold to determine whether the marginal processes jump to the same point or become slightly closer to each other. We call the coupling given by \eqref{basic-coup-3} the \emph{refined basic coupling} for pure jump L\'evy processes. By making full use of this coupling, we have obtained some new results for Wasserstein-type distances for SDEs with additive L\'evy noises, where L\'evy measure can be much singular. The reader can refer to \cite{Lwcw} for more details.

\subsubsection{{\bf Multiplicative L\'{e}vy noises.}}
For the SDE \eqref{s1} driven by multiplicative L\'{e}vy noises, the
jump system of the generater $L$ given by \eqref{proofth21} can be
simply understood as
$$x\longrightarrow x+\sigma(x)z, \qquad \nu(dz).$$ One may follow the construction of the refined basic coupling \eqref{basic-coup-3} above, and consider the following coupling (we can prove that this indeed associates with a coupling operator)
 \begin{equation*}
  (x,y)\longrightarrow
    \begin{cases}
    (x+\sigma(x)z, y+\sigma(y)(z+(x-y)_\kappa)), & \frac12 \mu_{(y-x)_\kappa}(dz);\\
    (x+\sigma(x)z, y+\sigma(y)(z+(y-x)_\kappa)), & \frac12 \mu_{(x-y)_\kappa}(dz);\\
    (x+\sigma(x)z, y+\sigma(y)z), & \big(\nu - \frac12 \mu_{(y-x)_\kappa}  -\frac12 \mu_{(x-y)_\kappa}\big)(dz).
    \end{cases}
  \end{equation*}
However, due to the appearance of diffusion coefficient $\sigma(x)$, we cannot  compare the distance after jump and the original distance from the coupling above. Actually, for coupling of SDEs with multiplicative L\'{e}vy noises, the situation becomes more complex, and we cannot  directly use the refined basic coupling. Roughly speaking, a reasonable and efficient coupling now should pay attention to the role of coefficient $\sigma(x)$.

Before moving further, we need some notation and elementary facts. Let $\Psi:\,\R ^d\rightarrow \R ^d$ be a continuous and bijective mapping, i.e., $\Psi$ is invertible and satisfies that $\Psi(\R ^d)=\R^d$. We further assume that $\Psi(0)\neq0$.
For any $n\ge1$, we define \begin{equation}\label{e:coumea}\mu_{\Psi}=\limsup_{n\to
\infty}\mu_{n,\Psi}:=\limsup_{n\to \infty}(\nu_n\wedge
(\nu_n{\Psi})),\end{equation} where $\nu_n(A)=\int_{A\cap\{|z|> 1/n\}}\,\nu(dz)$ and
$(\nu_n\Psi)(A)=\nu_n(\Psi(A))$ for all $A\in \mathscr{B}(\R^d)$.
The following observation is frequently used in the arguments below.
\begin{lemma}\label{C:mea} \begin{itemize}
\item[(1)] For any $A\in \mathscr{B}(\R^d)$,
$$(\mu_{\Psi}\Psi^{-1})(A)=\mu_{\Psi^{-1}}(A).$$
\item[(2)] Both $\mu_{\Psi}$ and $\mu_{\Psi^{-1}}$ are finite measures on $(\R^d, \mathscr{B}(\R^d))$.  \end{itemize}
\end{lemma}
\begin{proof} (1) Recall that for any two finite measures $\mu_1$ and $\mu_2$ on $(\R^d,\Bb(\R^d))$,
  $$\mu_1\wedge\mu_2:=\mu_1-(\mu_1-\mu_2)^+,$$
where $(\mu_1-\mu_2)^{+}$ and $(\mu_1-\mu_2)^{-}$ refer to the Jordan-Hahn decomposition of the signed measure $\mu_1-\mu_2$. In detail, for any $A\in\mathscr{B}(\R^d)$,
  $$(\mu_1-\mu_2)^{+}(A)=\sup\limits_{B\in \mathscr{B}(\R^d)}\{\mu_1(B)-\mu_2(B): B\subset A\}.$$

Note that $\nu_n$ and $\nu_n\Psi$ are finite measures on $(\R^d,\mathscr{B}(\R^d))$. By the definition of $\mu_{n,\Psi}$, for any $A\in \mathscr{B}(\R^d)$, we have
\begin{align*}
(\mu_{n,\Psi}\Psi^{-1})(A)&=\mu_{n,\Psi}(\Psi^{-1}(A))\\
&=[\nu_n\wedge (\nu_{n}\Psi)](\Psi^{-1}(A))=[(\nu_{n}\Psi)\wedge \nu_n](\Psi^{-1}(A))\\
&=(\nu_{n}\Psi)(\Psi^{-1}(A))-((\nu_{n}\Psi)-\nu_n)^+(\Psi^{-1}(A))\\
&=\nu_n(A)-\sup\limits_{B\in \mathscr{B}(\R^d)}\{(\nu_{n}\Psi)(B)-\nu_n(B):B\subset \Psi^{-1}(A)\}\\
&=\nu_n(A)-\sup\limits_{B\in \mathscr{B}(\R^d)}\{\nu_n(\Psi(B))-(\nu_n{\Psi^{-1}})(\Psi(B)):\Psi(B)\subset A\}\\
&=\nu_n(A)-\sup\limits_{\widetilde{B}\in \mathscr{B}(\R^d)}\{\nu_n(\widetilde{B})-(\nu_n\Psi^{-1})(\widetilde{B}):\widetilde{B}\subset A\}\\
&=\nu_n(A)-(\nu_n-(\nu_{n}\Psi^{-1}))^+(A)=(\nu_n\wedge(\nu_{n}\Psi^{-1}))(A)\\
&=\mu_{n,\Psi^{-1}}(A),
 \end{align*}
where in equalities above we used the fact that $\Psi$ is a
bijective mapping from $\R^d$ to $\R^d$. Then, the first required
assertion immediately follows from the equality above.

(2) Since $\Psi(0)\neq 0$ and $\Psi$ is continuous, there exists a
constant $\varepsilon_0>0$ such that $c_0:=\inf\{|\Psi(z)|:
|z|\le \varepsilon_0\}>0$. Thus, for any $n\ge1$,
\begin{align*}\int_{\R^d}(\mu_{n,\Psi})(dz)&=\int_{\R^d}\,(\nu_n\wedge (\nu_{n}\Psi))(dz)\\
&\le\int_{\{|z|> \varepsilon_0\}}\nu_n(dz)+\int_{\{|z|\le
\varepsilon_0\}}(\nu_{n}\Psi)(dz)\\
&\le \int_{\{|z|>\varepsilon_0\}}\nu_n(dz)+ \nu_n(\{\Psi(z):
|z|\le \varepsilon_0\})\\
&\le \nu(\{z\in \R^d: |z|> \varepsilon_0\})+\nu(\{z\in \R^d: |z|\ge
c_0\})\\
&=:c(c_0,\varepsilon_0)<\infty.
\end{align*} Letting $n\to \infty$, we get that
$\mu_{\Psi}(\R^d)\le c(c_0,\varepsilon_0)<\infty$. By (i) and the fact that $\Psi$ is
bijective, it also holds that $\mu_{\Psi^{-1}}(\R^d)<\infty$.
\end{proof}

Now, we consider the jump system  as follows:
  \begin{equation}\label{basic-coup-4}
  (x,y)\longrightarrow
    \begin{cases}
    (x+\sigma(x)z, y+\sigma(y)\Psi(z)), & \frac{1}{2} \mu_{\Psi}(dz);\\
    (x+\sigma(x)z, y+\sigma(y)\Psi^{-1}(z)), & \frac{1}{2} \mu_{\Psi^{-1}}(dz);\\
    (x+\sigma(x)z, y+\sigma(y)z), & \big(\nu - \frac{1}{2} \mu_{\Psi}  -\frac{1}{2} \mu_{\Psi^{-1}}\big)(dz).
    \end{cases}
  \end{equation}
More explicitly, for any $h\in C_b^2(\R^{d}\times\R^d)$ and $x,y\in \R^d$, we define
\begin{equation}\label{cp1}\begin{split}
\widetilde{L} h(x,y)=&\langle \nabla_xh(x,y),b(x)\rangle+\langle
\nabla_yh(x,y),b(y)\rangle\\
&+\frac{1}{2}\int\Big( h(x+\sigma(x)z,y+ \sigma(y)\Psi(z))-h(x,y)\\
&\qquad\qquad-\langle\nabla_xh(x,y), \sigma(x)z\rangle \I_{\{|z|\le
1\}}\\
     &\qquad\qquad-\langle\nabla_yh(x,y), \sigma(y)\Psi(z)\rangle \I_{\{|\Psi(z)|\le 1\}}\Big)\,\mu_{\Psi}(dz)\\
&+\frac{1}{2}\int\Big( h(x+\sigma(x)z,y+ \sigma(y)\Psi^{-1}(z))-h(x,y)\\
    &\qquad\qquad-\langle\nabla_xh(x,y), \sigma(x)z\rangle \I_{\{|z|\le 1\}}\\
    &\qquad\qquad-\langle\nabla_yh(x,y), \sigma(y)\Psi^{-1}(z)\rangle \I_{\{|\Psi^{-1}(z)|\le 1\}}\Big)\,\mu_{\Psi^{-1}}(dz)\\
&+\int\Big( h(x+\sigma(x)z,y+\sigma(y)z)-h(x,y)\\
&\qquad\qquad-\langle\nabla_xh(x,y), \sigma(x)z\rangle \I_{\{|z|\le
1\}}\\
    &\qquad\qquad-\langle\nabla_yh(x,y), \sigma(y)z\rangle \I_{\{|z|\le 1\}}\Big)\,\Big(\nu -\frac{1}{2}\mu_{\Psi} -\frac{1}{2}\mu_{\Psi^{-1}}\Big)(dz).\end{split}
 \end{equation}
By Lemma \ref{C:mea}(1), we can prove rigorously the following
statement.
\begin{theorem}\label{t:coup} The operator $ \widetilde{L}$ defined by \eqref{cp1} is a
coupling operator of the operator $L$ given by
\eqref{proofth21}.\end{theorem}
\begin{proof} First, let $h(x,y)=f(x)$ for any $x,y\in\R^d$, where $f\in C_b^2(\R^d)$.
Obviously, it holds that
\begin{align*}
 \widetilde{L} h(x,y)
    &=\langle \nabla f(x),b(x)\rangle\\
    &\quad+\frac{1}{2}\int\Big(f(x+\sigma(x)z)-f(x)-\langle\nabla f(x), \sigma(x)z\rangle \I_{\{|z|\le 1\}}\Big)\,\mu_{\Psi}(dz)\\
    &\quad+\frac{1}{2}\int\Big( f(x+\sigma(x)z)-f(x)-\langle\nabla f(x), \sigma(x)z\rangle \I_{\{|z|\le 1\}}\Big)\,\mu_{\Psi^{-1}}(dz)\\
    &\quad+\int\Big(f(x+\sigma(x)z)-f(x)-\langle\nabla f(x), \sigma(x)z\rangle \I_{\{|z|\le 1\}}\Big)\\
    &\qquad\qquad  \times \Big(\nu -\frac{1}{2}\mu_{\Psi} -\frac{1}{2}\mu_{\Psi^{-1}}\Big)(dz)\\
    &=Lf(x).
  \end{align*}

Secondly, let $h(x,y)=g(y)$ for any $x,y\in\R^d$, where $g\in
C_b^2(\R^d)$. Then, \begin{align*} \widetilde{L} h(x,y)
&=\langle \nabla g(y),b(y)\rangle\\
&\quad+ \frac{1}{2}\int \Big(g(y+ \sigma(y)\Psi(z))-g(y)\\
    &\qquad\qquad\qquad\qquad -\langle\nabla g(y), \sigma(y)\Psi(z)\rangle \I_{\{|\Psi(z)|\le
    1\}}\Big)\,\mu_{\Psi}(dz)\\
&\quad+\frac{1}{2}\int \Big( g(y+ \sigma(y)\Psi^{-1}(z))-g(y)\\
    &\qquad\qquad\qquad\qquad-\langle\nabla g(y), \sigma(y)\Psi^{-1}(z)\rangle \I_{\{|\Psi^{-1}(z)|\le 1\}}\Big)\,\mu_{\Psi^{-1}}(dz)\\
    &\quad+\int \Big(g(y+\sigma(y)z)-g(y)-\langle\nabla g(y), \sigma(y)z\rangle \I_{\{|z|\le 1\}}\Big)\\
    &\qquad\qquad\qquad\qquad\qquad \times\Big(\nu -\frac{1}{2}\mu_{\Psi} -\frac{1}{2}\mu_{\Psi^{-1}}\Big)(dz)\\
&= \langle \nabla g(y),b(y)\rangle\\
&\quad+ \frac{1}{2}\int \Big(g(y+ \sigma(y)\Psi(z))-g(y)\\
    &\qquad\qquad\qquad\qquad -\langle\nabla g(y), \sigma(y)\Psi(z)\rangle \I_{\{|\Psi(z)|\le
    1\}}\Big)\,(\mu_{\Psi^{-1}}\Psi)(dz)\\
&\quad+\frac{1}{2}\int \Big( g(y+ \sigma(y)\Psi^{-1}(z))-g(y)\\
    &\qquad\qquad\qquad\qquad-\langle\nabla g(y), \sigma(y)\Psi^{-1}(z)\rangle \I_{\{|\Psi^{-1}(z)|\le 1\}}\Big)\,(\mu_{\Psi}\Psi^{-1})(dz)\\
    &\quad+\int \Big(g(y+\sigma(y)z)-g(y)-\langle\nabla g(y), \sigma(y)z\rangle \I_{\{|z|\le 1\}}\Big)\\
    &\qquad\qquad\qquad\qquad\qquad \times\Big(\nu -\frac{1}{2}\mu_{\Psi} -\frac{1}{2}\mu_{\Psi^{-1}}\Big)(dz)\\
   & = \langle \nabla g(y),b(y)\rangle\\
&\quad+ \frac{1}{2}\int \Big(g(y+ \sigma(y)z)-g(y)\\
    &\qquad\qquad\qquad\qquad -\langle\nabla g(y), \sigma(y)z\rangle \I_{\{|z|\le
    1\}}\Big)\,\mu_{\Psi^{-1}}(dz)\\
&\quad+\frac{1}{2}\int \Big( g(y+ \sigma(y)z)-g(y)\\
    &\qquad\qquad\qquad\qquad-\langle\nabla g(y), \sigma(y)z\rangle \I_{\{|z|\le 1\}}\Big)\,\mu_{\Psi}(dz)\\
    &\quad+\int \Big(g(y+\sigma(y)z)-g(y)-\langle\nabla g(y), \sigma(y)z\rangle \I_{\{|z|\le 1\}}\Big)\\
    &\qquad\qquad\qquad\qquad\qquad \times\Big(\nu -\frac{1}{2}\mu_{\Psi} -\frac{1}{2}\mu_{\Psi^{-1}}\Big)(dz)\\
    &=Lg(y),\end{align*} where the second equality follows from Lemma \ref{C:mea}(1) and we used the measure transformations $\mu_{\Psi^{-1}}\Psi\mapsto
    \mu_{\Psi^{-1}}$ and $\mu_{\Psi}\Psi^{-1}\mapsto \mu_{\Psi}$ in the third equality.

     Combining both equalities above, we know that \eqref{marginality}
     holds true, and so the desired assertion follows.
\end{proof}

According to Theorem \ref{t:coup}, there exist a lot of
(non-trivial) coupling operators for the generator $L$ given by
\eqref{proofth21}. By the refined basic coupling
\eqref{basic-coup-3} (in particular the first row here), a proper
choice of $\Psi$ in \eqref{basic-coup-4} should satisfy that
$$x+\sigma(x)z=y+\sigma(y)\Psi(z)$$ for all $x,y\in\R^d$ with
$0<|x-y|\le \kappa$ for some constant $\kappa>0$. For this, in the remainder of this paper, we will
take
\begin{equation}\label{e:psi}\Psi(z)=\Psi_{\kappa,x,y}(z):=\sigma(y)^{-1}\big(\sigma(x)z+(x-y)_\kappa\big),\end{equation}
where $\kappa>0$ (which is a constant determinated later) and
$(x-y)_\kappa=\big(1\wedge \frac{\kappa}{|x-y|}\big)(x-y)$. \emph{Note
that, $\Psi(z)$ depends on  $\kappa$, $x$ and $y$, and for
simplicity we omit $\kappa,x,y$ in the notation.} Clearly,
$$\Psi^{-1}(z)=\sigma(x)^{-1}\big(\sigma(y)z-(x-y)_\kappa\big).$$ In
particular, with this choice, when $\sigma(x)={\rm I}_{d\times d}$ for
all $x\in \R^d$, \eqref{basic-coup-4} is reduced into
\eqref{basic-coup-3}. Moreover, by the nondegenerate property and
the continuity of $\sigma$, we know that for any $x,y\in\R^d$ with
$x\neq y$, $\Psi:\,\R ^d\rightarrow \R ^d$ is a continuous and
bijective mapping such that $\Psi(0)\neq 0$. In particular, Lemma
\ref{C:mea} applies.

\subsection{Coupling process}\label{secu1}
In this subsection, we prove the existence of the coupling process associated with the coupling operator $\widetilde{L}$ defined by \eqref{cp1}. We assume that the SDE \eqref{s1} has a unique strong solution.  By the L\'{e}vy--It\^{o} decomposition,
  $$Z_t=\int_0^t\int_{\{|z|>1\}}z\,N(ds,dz)+\int_0^t\int_{\{|z|\le1\}}z\,\tilde{N}(ds,dz),$$
where $N(ds,dz)$ is a Poisson random measure associated with $(Z_t)_{t\ge0}$, i.e.,\, $$N(ds,dz)=\sum_{\{0<{s'}\le s, \Delta Z_{s'}\neq 0\}}\delta_{(s', \Delta Z_{s'})}(ds,dz),$$ and
  $$\tilde{N}(ds,dz)=N(ds,dz)-ds\,\nu(dz)$$
is the corresponding compensated Poisson measure.
In order to write a coupling process explicitly, we extend the Poisson random measure $N$ from  $\R_+\times \R^d$ to $\R_+\times \R^d\times [0,1]$ in the following way \begin{equation*}\label{Poisson-meas}
  N(ds,dz,du) =\sum_{\{0<{s'}\le s, \Delta Z_{s'}\neq 0\}}\delta_{(s', \Delta Z_{s'})}(ds,dz)\I_{[0,1]}(du).
  \end{equation*}
and write
  $$Z_t=\int_0^t\int_{\R^d\times[0,1]} z\,\bar{N}(ds,dz,du),$$
where
  $$\bar{N}(ds,dz,du)=\I_{\{|z|>1\}\times[0,1]}N(ds,dz,du)+\I_{\{|z|\le1\}\times [0,1]}\tilde{N}(ds,dz,du).$$

Let $Z$ be a pure jump L\'evy process on $\R^d$ given above. We will construct a new L\'evy process $Z^*$ on $\R^d$
as follows. Suppose that a jump of $Z$ occurs at time $t$, and that
the process $Z$ moves from the point $Z_{t-}$ to $Z_{t-} + z$. Then,
we draw a random number $u\in [0,1]$ to determine whether the
process $Z^\ast$ should jump from the point $Z^\ast_{t-}$ to the
points $Z^\ast_{t-} +\Psi(z)$, $Z^\ast_{t-}+\Psi^{-1}(z)$ and
$Z^\ast_{t-}+z$, respectively. By taking into account the
characterization \eqref{basic-coup-4} for the coupling operator
$\widetilde L$ defined by \eqref{cp1}, the random number $u$ should
be determined by the following two factors:
  \begin{equation}\label{e:llle}\rho_\Psi(x,y,z)=\frac{\mu_{\Psi}(dz)}{\nu(dz)},\,\,\rho_{\Psi^{-1}}(x,y,z)=\frac{\mu_{\Psi^{-1}}(dz)}{\nu(dz)},\quad x,y,z\in \R^d.\end{equation} It is clear that both $\rho_\Psi(x,y,z)$ and $\rho_{\Psi^{-1}}(x,y,z)\in[0,1]$.
More explicitly, we will consider the system of equations:
  \begin{equation}\label{SDE-coup-eq-1}
  \begin{cases}
  dX_t=b(X_t)\, dt+\sigma(X_{t-})\,dZ_t,& X_0=x;\\
  dY_t= b(Y_t)\,dt+ \sigma(Y_{t-})\,dZ^\ast_t, & Y_0=y,
  \end{cases}
  \end{equation} where
  \begin{equation}\label{coup-SDE-2}
  \begin{split}
  d Z^\ast_t&= \int_{\R^d\times [0,1]} \Big[\Psi(z) \I_{\{u\le \frac12 \rho_{\Psi}(X_{t-},Y_{t-},z)\}} \\
  &\qquad \qquad \quad\,\,+\Psi^{-1}(z) \I_{\{\frac12 \rho_{\Psi}(X_{t-},Y_{t-},z)< u\le \frac12 [\rho_{\Psi}(X_{t-},Y_{t-},z)+\rho_{\Psi^{-1}}(X_{t-},Y_{t-},z)]\}}\\
  &\qquad\qquad\quad \,\, + z \I_{\{\frac12 [\rho_{\Psi}(X_{t-},Y_{t-},z)+\rho_{\Psi^{-1}}(X_{t-},Y_{t-},z)]< u\le 1\}}\Big]  \bar{N}(dt,dz,du)\\
  &\quad - \int_{\R^d\times [0,1]} \!
  \Big[ \Psi(z)\! \Big(\I_{\{|\Psi(z)|\le 1\}} \! -\!\I_{\{|z|\le 1\}}\Big)\! \I_{\{u\le \frac12 \rho_{\Psi}(X_{t-},Y_{t-},z)\}}\\
  &\qquad \qquad \quad\,\,\,\,+\Psi^{-1}(z)\Big(\I_{\{|\Psi^{-1}(z)\le 1\}} -\I_{\{|z|\le 1\}}\Big)\\
  &\qquad \qquad  \times \I_{\{\frac12 \rho_{\Psi}(X_{t-},Y_{t-},z)< u\le
  \frac12 [\rho_{\Psi}(X_{t-},Y_{t-},z)+\rho_{\Psi^{-1}}(X_{t-},Y_{t-},z)]\}}\Big] \,\nu(dz)\,du\,dt.
  \end{split}
  \end{equation} Note that, by Lemma \ref{C:mea}(2), $\mu_\Psi$ and $\mu_{\Psi^{-1}}$ are finite measures on $(\R^d,\mathscr{B}(\R^d))$, and so  \eqref{coup-SDE-2} is well defined.

  \begin{proposition} Suppose that the SDE \eqref{s1} has a unique strong solution. Then, the equation \eqref{SDE-coup-eq-1} also has a unique strong solution, denoted by $(X_t,Y_t)_{t\ge0}$, and the associated generator is just the coupling operator $\widetilde L$ given by \eqref{cp1}. In particular, $(X_t,Y_t)_{t\ge0}$ is a Markov coupling process for the unique strong solution to the SDE \eqref{s1}, and $X_t=Y_t$ for all $t\ge T$, where $T$ is the coupling time of $(X_t)_{t\ge0}$ and
$(Y_t)_{t\ge0}$, i.e.,\
  $T=\inf\{t\ge0: X_t=Y_t\}.$ \end{proposition}
\begin{proof}(1) We first simplify the formula \eqref{coup-SDE-2} for $Z^\ast$. We write \eqref{coup-SDE-2} as
\begin{align*}
  d Z^\ast_t&=\int_{\R^d\times[0,1]}z\,\bar{N}(dt,dz,du)\\
  &\quad+ \int_{\R^d\times [0,1]} \Big[(\Psi(z)-z) \I_{\{u\le \frac12 \rho_{\Psi}(X_{t-},Y_{t-},z)\}} \\
  &\quad\,\,\,\,+(\Psi^{-1}(z)-z) \I_{\{\frac12 \rho_{\Psi}(X_{t-},Y_{t-},z)< u\le \frac12 [\rho_{\Psi}(X_{t-},Y_{t-},z)+\rho_{\Psi^{-1}}(X_{t-},Y_{t-},z)]\}}\Big]  \bar{N}(dt,dz,du)\\
  &\quad - \int_{\R^d\times [0,1]} \!
  \Big[ \Psi(z)\! \Big(\I_{\{|\Psi(z)|\le 1\}} \! -\!\I_{\{|z|\le 1\}}\Big)\! \I_{\{u\le \frac12 \rho_{\Psi}(X_{t-},Y_{t-},z)\}}\\
  &\qquad \qquad \quad\,\,\,\,+\Psi^{-1}(z)\Big(\I_{\{|\Psi^{-1}(z)\le 1\}} -\I_{\{|z|\le 1\}}\Big)\\
  &\qquad \qquad \qquad\,\,\,\, \times \I_{\{\frac12 \rho_{\Psi}(X_{t-},Y_{t-},z)< u\le
  \frac12 [\rho_{\Psi}(X_{t-},Y_{t-},z)+\rho_{\Psi^{-1}}(X_{t-},Y_{t-},z)]\}}\Big] \,\nu(dz)\,du\,dt.
  \end{align*}

  According to Lemma \ref{C:mea}(1),
  \begin{align*}
   &\int_{\R^d\times [0,1]} \!
  \Big[ \Psi(z)\! \Big(\I_{\{|\Psi(z)|\le 1\}} \! -\!\I_{\{|z|\le 1\}}\Big)\! \I_{\{u\le \frac12 \rho_{\Psi}(X_{t-},Y_{t-},z)\}}\\
  &\qquad \qquad \quad\,\,\,\,+\Psi^{-1}(z)\Big(\I_{\{|\Psi^{-1}(z)\le 1\}} -\I_{\{|z|\le 1\}}\Big)\\
  &\qquad \qquad  \times \I_{\{\frac12 \rho_{\Psi}(X_{t-},Y_{t-},z)< u\le
  \frac12 [\rho_{\Psi}(X_{t-},Y_{t-},z)+\rho_{\Psi^{-1}}(X_{t-},Y_{t-},z)]\}}\Big] \,\nu(dz)\,du\\
  &=\frac{1}{2}\int_{\R^d} \Psi(z)\Big(\I_{\{|\Psi(z)|\le 1\}} \! -\!\I_{\{|z|\le 1\}}\Big)\,\mu_\Psi(dz)\\
  &\quad+\frac{1}{2}\int_{\R^d} \Psi^{-1}(z)\Big(\I_{\{|\Psi^{-1}(z)|\le 1\}} \! -\!\I_{\{|z|\le 1\}}\Big)\,\mu_{\Psi^{-1}}(dz)\\
  &=\frac{1}{2}\int_{\R^d}(z-\Psi^{-1}(z))\Big(\!\I_{\{|z|\le 1\}}-\I_{\{|\Psi^{-1}(z)|\le 1\}}\Big)\,\mu_{\Psi^{-1}}(dz)
  \end{align*} and
  \begin{align*}
    &\frac{1}{2}\int_{\R^d} (\Psi(z)-z)\I_{\{|z|\le1\}}\,\mu_\Psi(dz)=\frac{1}{2}\int_{\R^d} (z-\Psi^{-1}(z))\I_{\{|\Psi^{-1}(z)|\le1\}}\,\mu_{\Psi^{-1}}(dz).
  \end{align*}

  Combining all the equalities above together yields that
  \begin{align*}
  d Z^\ast_t&=\int_{\R^d\times[0,1]}z\,\bar{N}(dt,dz,du)\\
  &\quad+ \int_{\R^d\times [0,1]} \Big[(\Psi(z)-z) \I_{\{u\le \frac12 \rho_{\Psi}(X_{t-},Y_{t-},z)\}} \\
  &\quad\,\,+(\Psi^{-1}(z)-z) \I_{\{\frac12 \rho_{\Psi}(X_{t-},Y_{t-},z)< u\le \frac12 [\rho_{\Psi}(X_{t-},Y_{t-},z)+\rho_{\Psi^{-1}}(X_{t-},Y_{t-},z)]\}}\Big] {N}(dt,dz,du)\\
  &=:dZ_t+dG_t^*.
  \end{align*} Again, by Lemma \ref{C:mea}(2), $\mu_\Psi$ and $\mu_{\Psi^{-1}}$ are finite measures on $(\R^d,\mathscr{B}(\R^d))$, and so all the integrals above are well defined.
  In particular, we can rewrite \eqref{SDE-coup-eq-1} as follows
    \begin{equation}\label{SDE-coup-eq-444}
  \begin{cases}
  dX_t=b(X_t)\, dt+\sigma(X_{t-})\,dZ_t,& X_0=x;\\
  dY_t= b(Y_t)\,dt+ \sigma(Y_{t-})\,dZ_t+ \sigma(Y_{t-})\,dG^*_t, & Y_0=y.
  \end{cases}
  \end{equation}

 (2) We next follow the idea for the argument of \cite[Proposition 2.2]{Lwcw} and show that the SDE \eqref{SDE-coup-eq-444} has a unique strong solution. By assumption, the equation \eqref{s1} (i.e., the first equation in \eqref{SDE-coup-eq-444}) has a non-explosive and pathwise unique strong solution $(X_t)_{t\geq 0}$. We show that the sample paths of $(Y_t)_{t\geq 0}$ can be obtained by repeatedly modifying those of the solution of the following equation:
  \begin{equation}\label{2-prop-1.1}
  d \tilde Y_t=b(\tilde Y_t)\,dt+ \sigma(\tilde Y_{t-})dZ_t,\quad \tilde Y_0=y.
  \end{equation}

Denote by $(Y^{(1)}_t)_{t\ge0}$ the solution to \eqref{2-prop-1.1}. Take a uniformly distributed random variable $\zeta_1$ on $[0,1]$, and define the stopping times $T_1=\inf\big\{t>0: X_t=Y^{(1)}_t \big\}$ and
  \begin{align*}
  \sigma_1=\inf\bigg\{t>0: &\, \zeta_1\leq \frac12\Big(\rho_{\Psi}(X_{t},Y^{(1)}_{t},\Delta Z_t)+\rho_{\Psi^{-1}}(X_{t},Y^{(1)}_{t},\Delta Z_t)\Big)\bigg\}.
  \end{align*}
We consider two cases:
\begin{itemize}
\item[(i)] On the event $\{T_1\leq \sigma_1\}$, we set $Y_t=Y^{(1)}_t$ for all $t< T_1$; moreover, by the pathwise uniqueness of the equation \eqref{s1}, we can define $Y_t=X_t$ for $t\geq T_1$.
\item[(ii)] On the event $\{T_1> \sigma_1\}$, we define $Y_t=Y^{(1)}_t$ for all $t< \sigma_1$ and
  \begin{align*}Y_{\sigma_1}=&Y^{(1)}_{\sigma_1-}+ \sigma(Y^{(1)}_{\sigma_1-}) \Delta Z_{\sigma_1} \\
  &+ \begin{cases}
 \sigma(Y^{(1)}_{\sigma_1-})(\Psi(\Delta Z_{\sigma_1})-\Delta Z_{\sigma_1}), & \mbox{if } \zeta_1\leq \frac12 \rho_\Psi\big(X_{\sigma_1 -}, Y^{(1)}_{\sigma_1 -},\Delta Z_{\sigma_1} \big); \\
  \sigma(Y^{(1)}_{\sigma_1-})( \Psi^{-1}(\Delta Z_{\sigma_1})-\Delta Z_{\sigma_1}), & \mbox{if } \zeta_1> \frac12 \rho_\Psi\big(X_{\sigma_1 -}, Y^{(1)}_{\sigma_1 -},\Delta Z_{\sigma_1} \big).
  \end{cases}\end{align*}
\end{itemize}

Next, we restrict on the event $\{T_1> \sigma_1\}$ and consider the SDE \eqref{2-prop-1.1} with $t>\sigma_1$ and $\tilde Y_{\sigma_1}=Y_{\sigma_1}$. Denote its solution by $(Y^{(2)}_t)_{t\ge0}$. Similarly, we take another uniformly distributed random variable $\zeta_2$ on $[0,1]$, and define $T_2=\inf\big\{t>\sigma_1: X_t=Y^{(2)}_t\big\}$ and
  \begin{align*}
  \sigma_2=\inf\bigg\{t>\sigma_1: &\, \zeta_2\leq \frac12\Big(\rho_{\Psi}(X_{t},Y^{(2)}_{t},\Delta Z_t)+\rho_{\Psi^{-1}}(X_{t},Y^{(2)}_{t},\Delta Z_t)\Big)\bigg\}.
  \end{align*}
In the same way, we can define the process $(Y_t)_{t\ge0}$ till $t\leq \sigma_2$. We repeat this procedure and note that, thanks to the fact that $\mu_\Psi$ and $\mu_{\Psi^{-1}}$ are finite measures on $(\R^d,\mathscr{B}(\R^d))$ (by Lemma \ref{C:mea}(2)), only finite many modifications have to be made in any finite interval of time. Finally, we obtain the sample paths $(Y_t)_{t\geq 0}$.

(3) Denote by $(X_t,Y_t)_{t\ge0}$ the unique strong solution to \eqref{SDE-coup-eq-1}, and by $\bar{L}$ the associated Markov generator.
According to the It\^{o} formula, for any $f\in C_b^2(\R^{2d})$,
$\bar{L}f(x,y)$ enjoys the same formula as \eqref{cp1}; that is, the generator of the process $(X_t,Y_t)_{t\ge0}$ is just the coupling operator $\widetilde{L}$ defined by \eqref{cp1}. Thus, $(X_t,Y_t)_{t\ge0}$ is a Markov coupling of the process $X$ determined by the SDE \eqref{s1}.

When $X_{t-}=Y_{t-}$, $\Psi(z)=z$ and so $dZ^*_t=dZ_t$. Thus, by the Markov property of the process  $(X_t,Y_t)_{t\ge0}$ and the pathwise uniqueness of the SDE \eqref{s1}, $X_t=Y_t$ for any $t>T$, where $T$ is the coupling time of the process  $(X_t,Y_t)_{t\ge0}$.
\end{proof}

\section{Preliminary estimates on coupling operator}\label{Section3}
Let $ \widetilde{L} $ be the coupling operator defined by
\eqref{cp1}, where $\Psi$ is given by \eqref{e:psi}. Let $f\in
C([0,\infty))\cap C_b^2((0,\infty))$ such that $f(0)=0$, $f\ge0$, $f'\ge0$ and $f''\le 0$ on $(0,\infty)$. We
will give some estimates on $\widetilde{L}f(|x-y|)$.

According to \eqref{cp1}, we know that for any $f\in
C([0,\infty))\cap C_b^2((0,\infty))$ and any $x,y\in \R^d$ with $x\neq y$,
\begin{align*}
\widetilde{L} f(|x-y|)=&\frac{f'(|x-y|)}{|x-y|}\<
b(x)-b(y),x-y\>\\
&+\frac{1}{2}\int\Big(f\big(|(x+\sigma(x)z)-(y+\sigma(y)\Psi(z))|\big)-f(|x-y|)\\
    &\qquad\qquad -\frac{f'(|x-y|)}{|x-y|}\big\langle x-y,\sigma(x)z \big\rangle \I_{\{|z|\le 1\}}\\
    &\qquad\qquad + \frac{f'(|x-y|)}{|x-y|} \big\langle x-y,\sigma(y)\Psi(z) \big\rangle \I_{\{|\Psi(z)|\le 1\}}\Big)\mu_{\Psi}(dz)\\
    &+\frac{1}{2}\int\Big(f\big(|(x+\sigma(x)z)-(y+\sigma(y)\Psi^{-1}(z))|\big)-f(|x-y|)\\
    &\qquad\qquad -\frac{f'(|x-y|)}{|x-y|}\big\langle x-y,\sigma(x)z \big\rangle \I_{\{|z|\le 1\}}\\
    &\qquad\qquad+ \frac{f'(|x-y|)}{|x-y|} \big\langle x-y,\sigma(y)\Psi^{-1}(z) \big\rangle \I_{\{|\Psi^{-1}(z)|\le 1\}}\Big)\mu_{\Psi^{-1}}(dz)\\
    &+\int\Big(f\big(|(x+\sigma(x)z)-(y+\sigma(y)z)|\big)-f(|x-y|)\\
    &\qquad\quad- \frac{f'(|x-y|)}{|x-y|} \big\langle x-y,\sigma(x)z \big\rangle \I_{\{|z|\le 1\}}\\
    &\qquad\quad +\frac{f'(|x-y|)}{|x-y|}\big\langle x-y,\sigma(y)z \big\rangle \I_{\{|z|\le 1\}}\Big)\Big(\nu-\frac{1}{2}\mu_{\Psi}-\frac{1}{2}\mu_{\Psi^{-1}}\Big)(dz).
\end{align*}
By Lemma \ref{C:mea},
\begin{align*}
&\int\frac{f'(|x-y|)}{|x-y|} \big\langle x-y,\sigma(y)\Psi(z) \big\rangle \I_{\{|\Psi(z)|\le 1\}}\mu_{\Psi}(dz)\\
&=\int\frac{f'(|x-y|)}{|x-y|} \big\langle x-y,\sigma(y)\Psi(z) \big\rangle \I_{\{|\Psi(z)|\le 1\}}(\mu_{\Psi^{-1}}\Psi)(dz)\\
&=\int\frac{f'(|x-y|)}{|x-y|}\big\langle x-y,\sigma(y)z \big\rangle
\I_{\{|z|\le 1\}}\mu_{\Psi^{-1}}(dz),
 \end{align*}
 where we note that all the integrals above are well defined since both $\mu_{\Psi}$ and $\mu_{\Psi^{-1}}$ are finite measures, thanks to Lemma \ref{C:mea}(2).
Similarly, it also holds that
\begin{align*}
&\int\frac{f'(|x-y|)}{|x-y|} \big\langle x-y,\sigma(y)\Psi^{-1}(z) \big\rangle \I_{\{|\Psi^{-1}(z)|\le 1\}}\mu_{\Psi^{-1}}(dz)\\
&=\int\frac{f'(|x-y|)}{|x-y|}\big\langle x-y,\sigma(y)z \big\rangle
\I_{\{|z|\le 1\}}\mu_{\Psi}(dz).
 \end{align*}
Therefore, we arrive at for any $f\in
C([0,\infty))\cap C_b^2((0,\infty))$ and any $x,y\in \R^d$ with $x\neq y$,
\begin{align}\label{cpwd}\aligned
  &\widetilde L f(|x-y|)\\
   &= \frac{f'(|x-y|)}{|x-y|}\< b(x)-b(y),x-y\>\\
   &\quad-\frac{f'(|x-y|)}{2|x-y|}\Big\langle (\sigma(x)-\sigma(y))\int_{\{|z|\le 1\}}z(\mu_{\Psi}+\mu_{\Psi^{-1}})(dz),x-y \Big\rangle \\
    &\quad +\frac{1}{2}\int\Big(f\big(|(x+\sigma(x)z)-(y+\sigma(y)\Psi(z))|\big)-f(|x-y|)\Big)\mu_{\Psi}(dz)\\
   &\quad  +\frac{1}{2}\int\Big(f\big(|(x+\sigma(x)z)-(y+\sigma(y)\Psi^{-1}(z))|\big)-f(|x-y|)\Big)\mu_{\Psi^{-1}}(dz)\\
   &\quad  +\int\!\!\Big(f\big(|(x+\sigma(x)z)-(y+\sigma(y)z)|\big)-f(|x-y|)\\
   &\qquad\quad\,\,-\!\frac{f'(|x-y|)}{|x-y|}\big\langle x\!-\!y,(\sigma(x)\!-\!\sigma(y))z \big\rangle \I_{\{|z|\le 1\}}\Big)\!\Big(\nu\!-\!\frac{1}{2}\mu_{\Psi}\!-\!\frac{1}{2}\mu_{\Psi^{-1}}\Big)(dz)\\
    &=:I_1+I_2+I_3+I_4+I_5.\endaligned
  \end{align}

\begin{remark} We note that \eqref{cpwd} also can be directly deduced from \eqref{SDE-coup-eq-1}. Indeed, by \eqref{SDE-coup-eq-1} and \eqref{coup-SDE-2}, we have
\begin{align*}&d(X_t-Y_t)\\
&=(b(X_t)-b(Y_t))\,dt\\
&\quad+\int_{\R^d\times[0,1]} (\sigma(X_{t-})z-\sigma(Y_{t-})\Psi(z))\I_{\{u\le \frac12 \rho_{\Psi}(X_{t-},Y_{t-},z)\}}\,\bar{N}(dt,dz,du)\\
&\quad+\int_{\R^d\times[0,1]} (\sigma(X_{t-})z-\sigma(Y_{t-})\Psi^{-1}(z))\\
&\qquad\quad\times \I_{\{\frac12 \rho_{\Psi}(X_{t-},Y_{t-},z)< u\le
  \frac12 [\rho_{\Psi}(X_{t-},Y_{t-},z)+\rho_{\Psi^{-1}}(X_{t-},Y_{t-},z)]\}}\,\bar{N}(dt,dz,du)\\
&\quad+\int_{\R^d\times[0,1]} (\sigma(X_{t-})z-\sigma(Y_{t-})z)\I_{\{\frac12 [\rho_{\Psi}(X_{t-},Y_{t-},z)+\rho_{\Psi^{-1}}(X_{t-},Y_{t-},z)]<u\le 1\}}\,\bar{N}(dt,dz,du)\\
&\quad+ \int_{\R^d\times [0,1]} \!
  \Big[\sigma(Y_t) \Psi(z)\! \Big(\I_{\{|\Psi(z)|\le 1\}} \! -\!\I_{\{|z|\le 1\}}\Big)\! \I_{\{u\le \frac12 \rho_{\Psi}(X_{t-},Y_{t-},z)\}}\\
  &\qquad \qquad \quad\,\,\,\,+\sigma(Y_t)\Psi^{-1}(z)\Big(\I_{\{|\Psi^{-1}(z)\le 1\}} -\I_{\{|z|\le 1\}}\Big)\\
  &\qquad \qquad  \times \I_{\{\frac12 \rho_{\Psi}(X_{t-},Y_{t-},z)< u\le
  \frac12 [\rho_{\Psi}(X_{t-},Y_{t-},z)+\rho_{\Psi^{-1}}(X_{t-},Y_{t-},z)]\}}\Big] \,\nu(dz)\,du\,dt, \end{align*} where $\rho_\Psi(x,y,z)$ and $\rho_{\Psi^{-1}}(x,y,z)$ are given in \eqref{e:llle}. Then, by the It\^{o} formula, for any $f\in
C([0,\infty))\cap C_b^2((0,\infty))$ and any $x,y\in \R^d$ with $x\neq y$,
\begin{align*}
&\widetilde{L} f(|x-y|)\\
&=\frac{f'(|x-y|)}{|x-y|}\< b(x)-b(y),x-y\>\\
&\quad+\frac{1}{2}\int\Big(f(|x-y+\sigma(x)z-\sigma(y)\Psi(z)|)-f(|x-y|)\\
&\qquad\qquad\quad-\frac{f'(|x-y|)}{|x-y|}\langle x-y, \sigma(x)z-\sigma(y)\Psi(z)\rangle\I_{\{|z|\le 1\}}\Big)\,\mu_\Psi(dz)\\
&\quad+\frac{1}{2}\int\Big(f(|x-y+\sigma(x)z-\sigma(y)\Psi^{-1}(z)|)-f(|x-y|)\\
&\qquad\qquad\quad-\frac{f'(|x-y|)}{|x-y|}\langle x-y, \sigma(x)z-\sigma(y)\Psi^{-1}(z)\rangle\I_{\{|z|\le 1\}}\Big)\,\mu_{\Psi^{-1}}(dz)\\
&\quad+\int\Big(f(|x-y+\sigma(x)z-\sigma(y)z|)-f(|x-y|)\\
&\qquad\qquad-\frac{f'(|x-y|)}{|x-y|}\langle x-y, \sigma(x)z-\sigma(y)z\rangle\I_{\{|z|\le 1\}}\Big)\,\Big(\nu-\frac{1}{2}\mu_\Psi- \frac{1}{2}\mu_{\Psi^{-1}}\Big)(dz)\\
&\quad+\frac{f'(|x-y|)}{|x-y|}\int\langle x-y, \sigma(y)\Psi(z)\rangle\Big(\I_{\{|\Psi(z)|\le 1\}}-\I_{\{|z|\le 1\}}\Big)\,\mu_\Psi(dz)\\
&\quad+\frac{f'(|x-y|)}{|x-y|}\int\langle x-y, \sigma(y)\Psi^{-1}(z)\rangle\Big(\I_{\{|\Psi^{-1}(z)|\le 1\}}-\I_{\{|z|\le 1\}}\Big)\,\mu_{\Psi^{-1}}(dz)\\
&=\frac{f'(|x-y|)}{|x-y|}\< b(x)-b(y),x-y\>\\
&\quad+\frac{1}{2}\int\Big(f(|x-y+\sigma(x)z-\sigma(y)\Psi(z)|)-f(|x-y|)\Big)\,\mu_\Psi(dz)\\
&\quad+\frac{1}{2}\int\Big(f(|x-y+\sigma(x)z-\sigma(y)\Psi^{-1}(z)|)-f(|x-y|)\Big)\,\mu_{\Psi^{-1}}(dz)\\
&\quad+\int\Big(f(|x-y+\sigma(x)z-\sigma(y)z|)-f(|x-y|)\\
&\qquad\qquad-\frac{f'(|x-y|)}{|x-y|}\langle x-y, \sigma(x)z-\sigma(y)z\rangle\I_{\{|z|\le 1\}}\Big)\,\Big(\nu-\frac{1}{2}\mu_\Psi- \frac{1}{2}\mu_{\Psi^{-1}}\Big)(dz)\\
&\quad-\frac{f'(|x-y|)}{|x-y|}\int\langle x-y, \sigma(x)z\I_{\{|z|\le 1\}}-\sigma(y)\Psi(z)\I_{\{|\Psi(z)|\le 1\}}\rangle\Big)\,\mu_\Psi(dz)\\
&\quad-\frac{f'(|x-y|)}{|x-y|}\int\langle x-y, \sigma(x)z\I_{\{|z|\le 1\}}-\sigma(y)\Psi^{-1}(z)\I_{\{|\Psi^{-1}(z)|\le 1\}}\Big)\,\mu_{\Psi^{-1}}(dz). \end{align*} According to Lemma \ref{C:mea}, we know that
$$\int\langle x-y, \sigma(y)\Psi(z)\rangle\I_{\{|\Psi(z)|\le 1\}}\,\mu_\Psi(dz)=\int\langle x-y, \sigma(y)z\rangle\I_{\{|z|\le 1\}}\,\mu_{\Psi^{-1}}(dz)$$ and
$$\int\langle x-y, \sigma(y)\Psi^{-1}(z)\rangle\I_{\{|\Psi^{-1}(z)|\le 1\}}\,\mu_{\Psi^{-1}}(dz)=\int\langle x-y, \sigma(y)z\rangle\I_{\{|z|\le 1\}}\,\mu_{\Psi}(dz).$$ Hence, \eqref{cpwd} follows from all the equalities above.
  \end{remark}

Next, we assume that $f\in C([0,\infty))\cap C_b^2((0,\infty))$ such
that $f(0)=0$, $f\ge0$, $f'\ge0$ and $f''\le 0$ on $(0,\infty)$, and
will compute $I_i$ $(i=2,\cdots,5)$ in \eqref{cpwd} respectively.
Without loss of generality, under assumptions on $\sigma(x)$, in the
following we can assume that $$\Lambda^{-1}\le \inf_{x\in
\R^d}\left\{\|\sigma(x)\|_{{\rm H.S.}}\vee \|\sigma(x)^{-1}\|_{{\rm
H.S.}}\right\}\le  \sup_{x\in \R^d}\left\{\|\sigma(x)\|_{{\rm
H.S.}}\vee \|\sigma(x)^{-1}\|_{{\rm H.S.}}\right\}\le \Lambda.$$

(i) It is clear that
$$I_2\le \frac{1}{2}{f'(|x-y|)}\|\sigma(x)-\sigma(y)\|_{\mathrm{H.S.}}\int_{\{|z|\le 1\}}|z|(\mu_{\Psi}+\mu_{\Psi^{-1}})(dz).$$

(ii) By the definition of $\Psi$ in \eqref{e:psi}, we have
\begin{align*}
 I_3&=\frac{1}{2}\int\!\!\!\Big(f\big(|(x+\sigma(x)z)\!-\!\!(y+\sigma(y)\sigma(y)^{-1}(\sigma(x)z\!+\!(x-y)_\kappa))|\big)\!-\!f(|x-y|)\!\Big)\,\mu_{\Psi}(dz)\\
  &=\frac{1}{2}\int\Big(f\big(|x-y|-|x-y|\wedge\kappa\big)-f(|x-y|)\Big)\mu_{\Psi}(dz)\\
 &=\frac{1}{2}\mu_{\Psi}(\R ^d)\Big(f\big(|x-y|-|x-y|\wedge\kappa\big)-f(|x-y|)
 \Big),
  \end{align*} where in the last equality we used again the fact
  that $\mu_{\Psi}$ is a finite measure.

(iii) For any $R\in[1,\infty]$,
\begin{align*}
 I_4&=\frac{1}{2}\!\int\!\!\!\Big(f\!\big(|(x+\sigma(x)z)\!-\!(y+\sigma(y)\sigma(x)^{-1}(\sigma(y)z\!-\!(x-y)_\kappa))|\big)\!\!-\!\!f(|x\!-\!y|)\Big)\mu_{\Psi^{-1} }(dz)\\
 &=\frac{1}{2}\int\Big(f\big(|x-y|+|x-y|\wedge\kappa\big)-f(|x-y|)\Big)\mu_{\Psi^{-1}}(dz)\\
  &\quad +\frac{1}{2}\int\Big(f\big(|(x-y)+\sigma(y)\sigma(x)^{-1}(x-y)_\kappa+(\sigma(x)-\sigma(y)\sigma(x)^{-1}\sigma(y))z|\big)\\
 &\qquad \qquad\quad-f\big(|x-y|+|x-y|\wedge\kappa\big)\Big)\mu_{\Psi^{-1}}(dz)\\
 &=\frac{1}{2}\mu_{\Psi}(\R ^d)\Big(f\big(|x-y|+|x-y|\wedge\kappa\big)-f(|x-y|) \Big)\\
 &\quad +\frac{1}{2}\int_{\{|z|\leqslant R\}}\Big(f\big(|(x-y)+\sigma(y)\sigma(x)^{-1}(x-y)_\kappa+(\sigma(x)-\sigma(y)\sigma(x)^{-1}\sigma(y))z|\big)\\
 &\qquad \qquad\qquad\quad-f\big(|x-y|+|x-y|\wedge\kappa\big)\Big)\mu_{\Psi^{-1} }(dz)\\
 &\quad +\frac{1}{2}\int_{\{|z|> R\}}\Big(f\big(|(x-y)+\sigma(y)\sigma(x)^{-1}(x-y)_\kappa+(\sigma(x)-\sigma(y)\sigma(x)^{-1}\sigma(y))z|\big)\\
 &\qquad \qquad\qquad\quad-f\big(|x-y|+|x-y|\wedge\kappa\big)\Big)\mu_{\Psi^{-1} }(dz)\\
 &=:I_{4,1}+I_{4,2,R}+I_{4,3,R},
\end{align*} where in the third equality we used the fact that
$\mu_{\Psi}(\R^d)=\mu_{\Psi^{-1}}(\R^d)$, due to Lemma
\ref{C:mea}(1). By the elementary inequality
\begin{equation}\label{e:i1}f(a)-f(b)\leqslant f'(b)(a-b),\quad a\ge0,b>0,\end{equation} and the
fact that $f''\le0$ on $(0,\infty)$, we have
\begin{align*}
I_{4,2,R} &\le \frac{1}{2}f'(|x-y|+|x-y|\wedge\kappa)\int_{\{|z|\le
R\}}(|(\sigma(x)-\sigma(y)\sigma(x)^{-1}\sigma(y))z|\\
&\qquad
\qquad\qquad\qquad\qquad\qquad\qquad\quad\,\,\,\,+|(\sigma(y)\sigma(x)^{-1}\!-\!\mathrm{I}_{d\times
d})(x\!-\!y)_\kappa|)\,\mu_{\Psi^{-1}}(dz)\\
&\leqslant\frac{1}{2}f'(|x-y|)\|\sigma(x)-\sigma(y)\sigma(x)^{-1}\sigma(y)\|_{\mathrm{H.S.}}\int_{\{|z|\leqslant R\}}|z|\,\mu_{\Psi^{-1}}(dz)\\
 &\quad+\frac{1}{2}f'(|x-y|)\mu_{\Psi^{-1}}(\{z\in \R^d: |z|\leqslant R\})\|\sigma(y)\sigma(x)^{-1}-\mathrm{I}_{d\times d}\|_{\mathrm{H.S.}}(|x-y|\wedge\kappa)\\
 &\leqslant\frac{1}{2}f'(|x-y|)\big(1+\|\sigma(y)\sigma(x)^{-1}\|_{\mathrm{H.S.}}\big)\|\sigma(x)-\sigma(y)\|_{\mathrm{H.S.}}\int_{\{|z|\leqslant R\}}|z|\,
 \mu_{\Psi^{-1}}(dz)\\
 &\quad+\frac{1}{2}f'(|x-y|)\mu_{\Psi^{-1}}(\R^d)\|\sigma(x)-\sigma(y)\|_{\mathrm{H.S.}}\|\sigma(x)^{-1}\|_{\mathrm{H.S.}}(|x-y|\wedge\kappa),
\end{align*}
where in the last inequality we have used the facts that
\begin{equation}\label{I42inequ1}\begin{split}
\|\sigma(x)-\sigma(y)\sigma(x)^{-1}\sigma(y)\|_{\mathrm{H.S.}} &\le
\|\sigma(x)-\sigma(y)\|_{\mathrm{H.S.}}\\
&\quad+\|\sigma(y)-\sigma(y)\sigma(x)^{-1}\sigma(y)\|_{\mathrm{H.S.}}\\
&\le\|\sigma(x)-\sigma(y)\|_{\mathrm{H.S.}}\\
&\quad+\|\sigma(y)\sigma(x)^{-1}\|_{\mathrm{H.S.}}\|\sigma(x)-\sigma(y)\|_{\mathrm{H.S.}}\\
&=\big(1+\|\sigma(y)\sigma(x)^{-1}\|_{\mathrm{H.S.}}\big)\|\sigma(x)-\sigma(y)\|_{\mathrm{H.S.}}.
\end{split}\end{equation} and \begin{equation}\label{I42inequ2}\|\sigma(y)\sigma(x)^{-1}-\mathrm{I}_{d\times
d}\|_{\mathrm{H.S.}}\le
\|\sigma(x)-\sigma(y)\|_{\mathrm{H.S.}}\|\sigma(x)^{-1}\|_{\mathrm{H.S.}}.\end{equation}
On the other hand, since for all $a,b\ge0$, by $f''\le 0$ on $(0,\infty)$ and
$f(0)=0$,
\begin{equation}\label{e:i2}
f(a+b)-f(a)=\int_a^{a+b}f'(s)\,ds=\int_0^{b}f'(a+s)\,ds\leqslant\int_0^{b}f'(s)\,ds=f(b),\end{equation}
we have
\begin{align*}
I_{4,3,R}&\leqslant\frac{1}{2}\int_{\{|z|> R\}}\Big(f\big(|(\sigma(x)-\sigma(y)\sigma(x)^{-1}\sigma(y))z|\big)\\
&\qquad\qquad\qquad
+f\big(|x-y|+|\sigma(y)\sigma(x)^{-1}(x-y)_\kappa|\big)\\
 &\qquad\qquad\qquad
 -f\big(|x-y|+|x-y|\wedge\kappa\big)\Big)\,\mu_{\Psi^{-1}}(dz)\\
 &\leqslant\frac{1}{2}\int_{\{|z|> R\}}f\Big(\big(1+\|\sigma(y)\sigma(x)^{-1}\|_{\mathrm{H.S.}}\big)\|\sigma(x)-\sigma(y)\|_{\mathrm{H.S.}}|z|\Big)\mu_{\Psi^{-1}}(dz)\\
&\quad+\frac{1}{2}f'(|x-y|+|x-y|\wedge\kappa)|(\sigma(y)\sigma(x)^{-1}-\mathrm{I}_{d\times
d})(x-y)_\kappa|\\
&\qquad\times\mu_{\Psi^{-1}}\{z\in \R^d: |z|>
R\})\\
&\leqslant\frac{1}{2}\int_{\{|z|> R\}}f\Big(\big(1+\|\sigma(y)\sigma(x)^{-1}\|_{\mathrm{H.S.}}\big)\|\sigma(x)-\sigma(y)\|_{\mathrm{H.S.}}|z|\Big)\mu_{\Psi^{-1}}(dz)\\
&\quad+\frac{1}{2}f'(|x-y|)\mu_{\Psi}(\R^d)\|\sigma(x)-\sigma(y)\|_{\mathrm{H.S.}}\|\sigma(x)^{-1}\|_{\mathrm{H.S.}}(|x-y|\wedge\kappa),
\end{align*}
where in the second inequality we used \eqref{I42inequ1}, and the
last one follows from \eqref{I42inequ2} and the facts that $f''\le
0$ on $(0,\infty)$ and $\mu_{\Psi}(\R^d)=\mu_{\Psi^{-1}}(\R^d)$.

Combining all the conclusions above, we obtain that
\begin{align*}
 I_4&\le\frac{1}{2}\mu_{\Psi}(\R ^d)\Big(f\big(|x-y|+|x-y|\wedge\kappa\big)-f(|x-y|) \Big)\\
 &\quad+ \frac{1}{2}f'(|x-y|)\|\sigma(x)-\sigma(y)\|_{\mathrm{H.S.}}\Bigg[2\mu_{\Psi^{-1}}(\R^d)\|\sigma(x)^{-1}\|_{\mathrm{H.S.}}(|x-y|\wedge\kappa)\\
 &\qquad\qquad\qquad\qquad\qquad\qquad+\big(1+\|\sigma(y)\sigma(x)^{-1}\|_{\mathrm{H.S.}}\big)\int_{\{|z|\leqslant R\}}|z|
 \mu_{\Psi^{-1}}(dz)\Bigg]\\
 &\quad+\frac{1}{2}\int_{\{|z|> R\}}f\Big(\big(1+\|\sigma(y)\sigma(x)^{-1}\|_{\mathrm{H.S.}}\big)\|\sigma(x)-\sigma(y)\|_{\mathrm{H.S.}}|z|\Big)\mu_{\Psi^{-1}}(dz).\end{align*}

 (iv) For $I_5$,
we have
 \begin{align*}
  I_5&=\int_{\{|z|\le 1\}}\Big(f\big(|x-y+(\sigma(x)-\sigma(y))z|\big)-f(|x-y|)\\
    &\hskip60pt - \frac{f'(|x-y|)}{|x-y|}\langle x-y, (\sigma(x)-\sigma(y))z\rangle \Big)\,\Big(\nu\!-\!\frac{1}{2}\mu_{\Psi}\!-\!\frac{1}{2}\mu_{\Psi^{-1}}\Big)(dz)\\
    &\hskip17pt+\int_{\{|z|>1\}}\!\!\Big(f\big(|x-y+(\sigma(x)-\sigma(y))z|\big)-f(|x-y|)\Big)\,\Big(\nu-\frac{1}{2}\mu_{\Psi}-\frac{1}{2}\mu_{\Psi^{-1}}\Big)(dz)\\
    &=:I_{5,1}+I_{5,2}.\end{align*} By \eqref{e:i1} again, we find
\begin{align*}
  I_{5,1}&\leqslant f'(|x-y|)\int_{\{|z|\le 1\}}\Big(|x-y+(\sigma(x)-\sigma(y))z|-|x-y|\\
    &\hskip80pt -\frac{1}{|x-y|}\langle x-y,
    (\sigma(x)\!-\!\sigma(y))z\rangle\Big)\,\Big(\nu\!-\!\frac{1}{2}\mu_{\Psi}\!-\!\frac{1}{2}\mu_{\Psi^{-1}}(dz)\Big)\\
    &\leqslant \frac{f'(|x-y|)}{2|x-y|}\int_{\{|z|\le 1\}}|(\sigma(x)-\sigma(y))z|^2\,\Big(\nu-\frac{1}{2}\mu_{\Psi}-\frac{1}{2}\mu_{\Psi^{-1}}\Big)(dz)\\
    &\leqslant f'(|x-y|)\frac{\|\sigma(x)-\sigma(y)\|_{\mathrm{H.S.}}^2}{2|x-y|}\int_{\{|z|\le 1\}}|z|^2\,\Big(\nu-\frac{1}{2}\mu_{\Psi}-\frac{1}{2}\mu_{\Psi^{-1}}\Big)(dz)\\
        &\leqslant f'(|x-y|)\frac{\|\sigma(x)-\sigma(y)\|_{\mathrm{H.S.}}^2}{2|x-y|}\int_{\{|z|\le 1\}}|z|^2\,\nu(dz),
   \end{align*}
where in the second inequality we used the fact that
  $$|x+y|-|x|- \frac{1}{|x|}\langle x, y\rangle\le \frac{\,\,\,|y|^2}{2|x|},\quad x,y\in\R^d\,\, {\rm with } \,\,x\neq 0.$$
On the other hand, using \eqref{e:i1} and \eqref{e:i2}, we obtain
that for all $R\in[1,\infty]$,
\begin{align*}
 I_{5,2}
 &=\int_{\{1<|z|\le R\}}\Big(f\big(|x-y+(\sigma(x)-\sigma(y))z|\big)\!-\!f(|x-y|) \Big)\,\!\!\Big(\nu-\frac{1}{2}\mu_{\Psi}-\frac{1}{2}\mu_{\Psi^{-1}}\Big)(dz)\\
 &\quad+\int_{\{|z|>R\}}\Big(f\big(|x-y+(\sigma(x)-\sigma(y))z|\big)\!-\!f(|x-y|) \Big)\,\!\!\Big(\nu-\frac{1}{2}\mu_{\Psi}-\frac{1}{2}\mu_{\Psi^{-1}}\Big)(dz)\\
 &\le f'(|x-y|)\|\sigma(x)-\sigma(y)\|_{\mathrm{H.S.}}\int_{\{1<|z|\le R\}}|z|\,\Big(\nu-\frac{1}{2}\mu_{\Psi}-\frac{1}{2}\mu_{\Psi^{-1}}\Big)(dz)\\
 &\quad+ \int_{\{|z|> R\}}f(\|\sigma(x)-\sigma(y)\|_{\mathrm{H.S.}}|z|)\,\Big(\nu-\frac{1}{2}\mu_{\Psi}-\frac{1}{2}\mu_{\Psi^{-1}}\Big)(dz)\\
 &\le f'(|x-y|)\|\sigma(x)-\sigma(y)\|_{\mathrm{H.S.}}\int_{\{1<|z|\le R\}}|z|\,\nu(dz)\\
 &\quad+ \int_{\{|z|> R\}}f(\|\sigma(x)-\sigma(y)\|_{\mathrm{H.S.}}|z|)\,\nu(dz).
\end{align*}

Combining both inequalities above, we arrive at
 \begin{align*}
  I_5&\le f'(|x-y|)\|\sigma(x)-\sigma(y)\|_{\mathrm{H.S.}}\\
  &\quad\,\,\,\,\times \bigg[\frac{\|\sigma(x)-\sigma(y)\|_{\mathrm{H.S.}}}{2|x-y|}\int_{\{|z|\le 1\}}|z|^2\,\nu(dz)+\int_{\{1<|z|\le R\}}|z|\,\nu(dz)\bigg]\\
  &\quad+\int_{\{|z|> R\}}f(\|\sigma(x)-\sigma(y)\|_{\mathrm{H.S.}}|z|)\,\nu(dz).\end{align*}

\medskip

Finally, putting all the estimates in (i)--(iv) into \eqref{cpwd}, we
can get the following statement.

\begin{proposition}\label{P:estimates} Assume that \eqref{p:est0}
holds. Then, for all $f\in
C([0,\infty))\cap C_b^2((0,\infty))$ such that $f(0)=0$, $f\ge0$, $f'\ge0$ and $f''\le 0$ on $(0,\infty)$, any $R\in [1,\infty]$ and
$x,y\in \R^d$ with $x\neq y$,
\begin{equation}\label{p:est1}\begin{split}\widetilde L f(|x-y|)\le& \Theta_0(f)(x,y)+ \frac{f'(|x-y|)}{|x-y|}\< b(x)-b(y),x-y\>\\
&+ f'(|x-y|)\|\sigma(x)-\sigma(y)\|_{\mathrm{H.S.}}\Theta_{\le
R}(x,y)\\
&+\Theta_{> R}(f)(x,y),\end{split}\end{equation} where
\begin{align*}\Theta_0(f)(x,y):&= \frac{1}{2}\mu_{\Psi}(\R
^d)\Big(f\big(|x-y|+|x-y|\wedge\kappa\big)\\
&\qquad\qquad\quad\,\,+f\big(|x-y|-|x-y|\wedge\kappa\big)-2f(|x-y|)
\Big),\\
\Theta_{\le R}(x,y):&=\Lambda \mu_{\Psi}(\R^d)(|x-y|\wedge
\kappa)+\frac{\|\sigma(x)-\sigma(y)\|_{\mathrm{H.S.}}}{2|x-y|}\int_{\{|z|\le 1\}}|z|^2\,\nu(dz)\\
&\quad+(1+\Lambda^2/2)\int_{\{|z|\le
R\}}|z|\,(\mu_{\Psi}+\mu_{\Psi^{-1}})(dz)+\int_{\{1<|z|\le
R\}}|z|\,\nu(dz)\end{align*} and
$$\Theta_{> R}(f)(x,y):=2\int_{\{|z|> R\}}f\Big(\big(1+\Lambda^2\big)\|\sigma(x)-\sigma(y)\|_{\mathrm{H.S.}}|z|\Big)
\nu(dz).$$
\end{proposition}

\begin{remark} We should mention that, if $\nu$ in the definitions of $\mu_\Psi$ and $\mu_{\Psi^{-1}}$ is replaced by  any Borel measure $\nu_0$ on $(\R^d,\mathscr{B}(\R^d))$ such that $0<\nu_0\le \nu$, then all the conclusions above still hold true.
\end{remark}
 \section{Regularity and ergodicity via coupling}\label{section44}
Assume that the SDE \eqref{s1} has a unique strong solution, which is denoted by $X:=(X_t)_{t\ge0}$. Let $(P_t)_{t\ge0}$ be the associated semigroup. Let $\widetilde L$ be the coupling operator given by \eqref{cp1}.

\subsection{Regularity via coupling}
The statement below shows an idea to establish regularity properties of semigroups by adopting the coupling operator $\widetilde{L}$.

\begin{proposition}\label{P:re}Assume that there exist a constant $\varepsilon_0>0$ and a sequence of positive and increasing functions $\{\psi_n\}_{n\ge1}$ such that for all $x,y\in \R^d$ with $1/n\le |x-y|\le \varepsilon\le \varepsilon_0$,
\begin{equation}\label{e:sregu}\widetilde{L} \psi_n(|x-y|)\le -\lambda(\varepsilon),\end{equation} where $\lambda(\varepsilon)$ is a positive constant independent of $n$. Then, for
any $t>0$ and $f\in B_b(\R^d)$,
$$\sup_{x\neq y}\frac{{|P_t f(x)-P_t f(y)|}}{\psi_\infty(|x-y|)}\le
  2\|f\|_\infty \sup_{\varepsilon\in(0,\varepsilon_0]}\bigg[\frac{1}{\psi_\infty(\varepsilon)}+\frac{1}{t\lambda(\varepsilon)}\bigg]. $$ where
$\psi_\infty=\liminf_{n\to\infty}\psi_n.$
\end{proposition}

\begin{proof} The proof was almost known before, e.g.,\ see that of \cite[Theorem 5.1]{Lwcw} or \cite[Theorem 1.2]{LW14}. For the sake of completeness, we present it here. Let $(X_t,Y_t)_{t\ge0}$ be the coupling process  constructed
in Subsection \ref{secu1}, and denote by $\widetilde{\Pp}^{(x,y)}$ and
$\widetilde{\Ee}^{(x,y)}$ the distribution and the expectation of
$(X_t,Y_t)_{t\ge0}$ starting from $(x,y)$, respectively. For any $\varepsilon\in(0,\varepsilon_0]$ and
$n\ge1$, we set
  $$\aligned S_\varepsilon&:=\inf\{t\ge0: |X_t-Y_t|>\varepsilon\},\\
  T_n&:=\inf\{t\ge0: |X_t-Y_t|\le 1/n\}.\endaligned$$
Note that $T_n\uparrow T$ as $n\uparrow\infty,$ where $T$ is the coupling time, i.e.,\ $$T:=\inf\{t\ge0: X_t=Y_t\}.$$

For any $x,$ $y\in\R^d$ with $0<|x-y|<\varepsilon\le \varepsilon_0$, we take $n$ large enough such that $|x-y|>1/n$. Then, by \eqref{e:sregu}, for any $t>0$,
  $$\aligned
  0&\le\widetilde{\Ee}^{(x,y)}\psi_n\big(|X_{t\wedge T_n\wedge S_\varepsilon}-Y_{t\wedge T_n\wedge S_\varepsilon}|\big)\\
  &=\psi_n(|x-y|)+\widetilde{\Ee}^{(x,y)}\bigg(\int_0^{t\wedge T_n\wedge S_\varepsilon} \widetilde{L} \psi_n\big(|X_{s}-Y_{s}|\big)\,ds\bigg)\\
  &\le \psi_n(|x-y|)-\lambda(\varepsilon)\widetilde{\Ee}^{(x,y)}(t\wedge
  T_n\wedge S_\varepsilon).\endaligned$$
Hence,
  $$\widetilde{\Ee}^{(x,y)}(t\wedge T_n\wedge S_\varepsilon)\le
  \frac{ \psi_n(|x-y|)}{\lambda(\varepsilon)}.$$ Letting $t\to\infty$,
  $$\widetilde{\Ee}^{(x,y)}(T_n\wedge S_\varepsilon)\le
  \frac{ \psi_n(|x-y|)}{\lambda(\varepsilon)}.$$

On the other hand, again by \eqref{e:sregu}, for any $x$, $y\in\R^d$ with $1/n\le |x-y|<\varepsilon\le \varepsilon_0$,
  $$\aligned &\widetilde{\Ee}^{(x,y)}\psi_n\big(|X_{t\wedge T_n\wedge S_\varepsilon}-Y_{t\wedge T_n\wedge
  S_\varepsilon}|\big)\\
  &=\psi_n(|x-y|)+\widetilde{\Ee}^{(x,y)}\bigg(\int_0^{t\wedge T_n\wedge S_\varepsilon}  \widetilde{L} \psi_n(| X_u-Y_u|)\,du\bigg)\\
  &\le \psi_n(|x-y|).\endaligned$$
This along with the increasing property of $\psi_n$ yields that
  $$\widetilde{\Pp}^{(x,y)}(S_\varepsilon<T_n\wedge t)\le \frac{\psi_n(|x-y|)}{\psi_n(\varepsilon)}.$$
Letting $t\to \infty$, $$\widetilde{\Pp}^{(x,y)}(S_\varepsilon<T_n)\le \frac{\psi_n(|x-y|)}{\psi_n(\varepsilon)}.$$

Therefore, combining both estimates above, we obtain that
  $$\aligned
   \widetilde{\Pp}^{(x,y)}(T_n> t)
    &\le \widetilde{\Pp}^{{(x,y)}}(T_n\wedge S_\varepsilon>t)+\widetilde{\Pp}^{{(x,y)}}(T_n>S_\varepsilon)\\
  &\le \frac{\widetilde{\Ee}^{(x,y)}(T_n\wedge S_\varepsilon)}{t}+\frac{\psi_n(|x-y|)}{\psi_n(\varepsilon)}\\
  &\le \psi_n(|x-y|)\bigg[\frac{1}{\psi_n(\varepsilon)}+\frac{1}{t\lambda(\varepsilon)}\bigg].
  \endaligned$$ It follows that
  \begin{align*}\widetilde{\Pp}^{(x,y)}(T> t)&= \lim_{n\to\infty}\widetilde{\Pp}^{(x,y)}(T_n> t)\\
  &\le \liminf_{n\to\infty}\left\{\psi_n(|x-y|)\bigg[\frac{1}{\psi_n(\varepsilon)}+\frac{1}{t\lambda(\varepsilon)}\bigg]\right\}\\
  &\le  \psi_\infty(|x-y|)\bigg[\frac{1}{\psi_\infty(\varepsilon)}+\frac{1}{t\lambda(\varepsilon)}\bigg].\end{align*}
Thus, for any $f\in B_b(\R^d)$, $t>0$ and any $x$, $y\in\R^d$ with $0<|x-y|<\varepsilon\le \varepsilon_0$,
  \begin{align*}
 {|P_t f(x)-P_t f(y)|}&={|\Ee^xf(X_t)-\Ee^yf(Y_t)|}\\
  &={\big|\widetilde{{\Ee}}^{(x,y)}(f(X_t)-f(Y_t))\big|}\\
  &={\big|\widetilde{{\Ee}}^{(x,y)}(f(X_t)-f(Y_t))\I_{\{T> t\}}\big|}\\
  &\le 2\|f\|_\infty \widetilde{{\Pp}}^{(x,y)}(T> t)\\
  &\le 2\|f\|_\infty  \psi_\infty(|x-y|)\bigg[\frac{1}{\psi_\infty(\varepsilon)}+\frac{1}{t\lambda(\varepsilon)}\bigg].
  \end{align*}
Consequently,
  $$\sup_{|x-y|\le\varepsilon}\frac{{|P_t f(x)-P_t f(y)|}}{\psi_\infty(|x-y|)}
  \le 2\|f\|_\infty\bigg[\frac{1}{\psi_\infty(\varepsilon)}+\frac{1}{t\lambda(\varepsilon)}\bigg].$$
Since $\psi_\infty$ is increasing on $(0,\infty)$, and
  $$\sup_{|x-y|\ge\varepsilon}\frac{{|P_t f(x)-P_t f(y)|}}{\psi_\infty(|x-y|)}   \le \frac{2\|f\|_\infty}{\psi_\infty(\varepsilon)},$$
we further obtain that for all $\varepsilon\in(0,\varepsilon_0]$,
  $$\sup_{x\neq y}\frac{{|P_t f(x)-P_t f(y)|}}{\psi_\infty(|x-y|)}\le
  2\|f\|_\infty \bigg[\frac{1}{\psi_\infty(\varepsilon)}+\frac{1}{t\lambda(\varepsilon)}\bigg]. $$
Taking infimum with respect to $\varepsilon\in(0,\varepsilon_0]$ in
the right hand side of the inequality above proves the desired
assertion.
\end{proof}
\begin{theorem}\label{T:re} Suppose that the diffusion coefficient $\sigma(x)$ is Lipschitz continuous and satisfies \eqref{p:est0}, and the drift term $b(x)$ is locally $\beta$-H\"{o}lder
continuous with $\beta\in(0,1]$. Assume also that there is  a
nonnegative and $C_b([0,\infty))\cap C^3((0,\infty))$-function $\psi$ such that
\begin{itemize}
\item[(i)] $\psi(0)=0$, $\psi'\ge0$, $\psi''\le0$ and $\psi'''\ge0$ on $(0,2]$;
\item[(ii)] For any constants $c_1,c_2>0$,
\begin{equation}\label{e:fe1}\limsup_{r\to0}\bigg[J(r)r^2\psi''(2r)+c_1K(r)\psi'(r)r+c_2\psi'(r)r^\beta\bigg]<0,\end{equation} where \begin{equation}\label{e:llev}
J(r)=\inf_{x,y\in\R^d:|x-y|\le r}\mu_\Psi(\R^d)\end{equation} and
$$K(r)=\sup_{x,y\in\R^d:|x-y|=
r}\Big(\mu_\Psi(\R^d)|x-y|+\int_{\{|z|\le
2\}}|z|\,(\mu_{\Psi}+\mu_{\Psi^{-1}})(dz)\Big).$$
\end{itemize}Then, there are constants $C,\varepsilon_0>0$ such that for all $f\in B_b(\R^d)$ and $t>0$,
$$\sup_{x\neq y}\frac{|P_tf(x)-P_tf(y)|}{|x-y|}\le C\|f\|_\infty\inf_{\varepsilon\in(0,\varepsilon_0]}\bigg[\frac{1}{\psi(\varepsilon)}+\frac{1}{t\lambda_\psi(\varepsilon)}\bigg],$$ where
$$\lambda_\psi(\varepsilon)=-\sup_{0<r\le \varepsilon}J(r)r^2\psi''(2r).$$\end{theorem}
\begin{proof}It is clear that \eqref{e:fe1} implies
\begin{equation}\label{e:ffkk2} \limsup_{r\to0}J(r)r^2\psi''(2r)<0.\end{equation}

Let $\varepsilon\in(0,\kappa\wedge1)$. For any $x,y\in \R^d$ with
$0<|x-y|\le \varepsilon$, by \eqref{p:est1} (with $R=2$) and the
assumptions that $\sigma(x)$ is Lipschitz continuous and satisfies
\eqref{p:est0}, and $b(x)$ is locally $\beta$-H\"{o}lder continuous with
$\beta\in(0,1]$, we find that
\begin{align*}\widetilde{L}\psi(|x-y|)&\le \frac{1}{2}\mu_\Psi(\R^d)\big(\psi(2|x-y|)-2\psi(|x-y|)\big)\\
&\quad+c_1\psi'(|x-y|)|x-y|\Big(\mu_\Psi(\R^d)|x-y|+\int_{\{|z|\le 2\}}|z|\,(\mu_{\Psi}+\mu_{\Psi^{-1}})(dz)\Big)\\
&\quad+c_2\psi'(|x-y|)|x-y|^\beta+c_3\int_{\{|z|\ge2\}}\psi(c_4|x-y||z|)\,\nu(dz)
\end{align*} for some constants $c_1,c_2,c_3,c_4>0$.
Note that, by $\psi'''\geq 0$ on $(0,2]$ and $\psi(0)=0$, we have for $r>0$ small enough,
  $$\psi(2r)-2\psi(r)=\int_0^r\int_s^{r+s} \psi''(u)\,du\,ds\le\psi''(2r) r^2.$$
Then, using \eqref{e:ffkk2}  and \eqref{e:fe1}, we
can choose $\varepsilon_0\in(0,\kappa \wedge1)$ such that for all
$x,y\in\R^d$ with $0<|x-y|\le \varepsilon\le \varepsilon_0$,
 \begin{align*}\widetilde{L}\psi(|x-y|)&\le \frac{1}{2}J(|x-y|)\psi''(2|x-y|)|x-y|^2+c_1K(|x-y|)\psi'(|x-y|)|x-y|\\
&\quad+c_2\psi'(|x-y|)|x-y|^\beta+c_3\int_{\{|z|\ge2\}}\psi(c_4|x-y||z|)\,\nu(dz)\\
&\le \frac{1}{4}\bigg[J(|x-y|)\psi''(2|x-y|)|x-y|^2+c_1K(|x-y|)\psi'(|x-y|)|x-y|\\
&\quad\quad+c_2\psi'(|x-y|)|x-y|^\beta\bigg]\\
&\le \frac{1}{4}\sup_{|x-y|\le \varepsilon}\bigg[J(|x-y|)\psi''(2|x-y|)|x-y|^2\\
&\qquad\qquad\quad+c_1K(|x-y|)\psi'(|x-y|)|x-y|+c_2\psi'(|x-y|)|x-y|^\beta\bigg].\end{align*}
where in the first inequality the constant $c_2$ may depend on $\varepsilon_0$ but can be chosen to be independent of $\varepsilon$, and  in the second inequality we also used the fact that
$$\lim_{|x-y|\to0} \int_{\{|z|\ge2\}}\psi(c_4|x-y||z|)\,\nu(dz)=0.$$
By \eqref{e:ffkk2} and \eqref{e:fe1} again, we furthermore get (by possibly choosing
$\varepsilon_0$ small enough) that
 \begin{align*}
\widetilde{L}\psi(|x-y|)&\le \frac{1}{8}\sup_{|x-y|\le \varepsilon}J(|x-y|)\psi''(2|x-y|)|x-y|^2=- \frac{1}{8}\lambda_\psi(\varepsilon)<0,
\end{align*}

Having the inequality above at hand, we can obtain the desired assertion by applying Proposition \ref{P:re}.
\end{proof}

\subsection{Ergodicity via coupling}
The following proposition is essentially taken from \cite[Theorem
3.1]{Lwcw}.
\begin{proposition}\label{P:rec}Assume that there exist  a constant $\lambda>0$ and a sequence of positive functions $\{\psi_n\}_{n\ge1}$ such that for all $x,y\in \R^d$ with $1/n\le |x-y|\le n$,
$$\widetilde{L} \psi_n(|x-y|)\le - \lambda  \psi_n(|x-y|).$$ Then, for
any $t>0$ and $x,y\in \R^d$,
$$W_{\psi_\infty}(\delta_x P_t, \delta_y P_t)\le
\psi_\infty(|x-y|)e^{-\lambda t},$$ where
$\psi_\infty=\liminf_{n\to\infty}\psi_n.$
\end{proposition}\begin{proof}The proof is inspired by that of \cite[Theorem 1.3]{LW} or \cite[Theorem 1.2]{Wang}. We can refer to step 2 in the proof of
\cite[Theorem 3.1]{Lwcw} for the details. \end{proof}

Next, we assume that \eqref{e:ffgg} holds. This implies
that we can take $R=\infty$ in the estimate \eqref{p:est1}. Motivated by \cite[Theorems 4.2 and 4.4]{Lwcw}, we have the following statements.
\begin{theorem}\label{thtpw}
Assume that the diffusion coefficient $\sigma(x)$ is Lipschitz
continuous with Lipschiz constant $L_\sigma>0$ and satisfies
\eqref{p:est0} with some constant $\Lambda>0$, and that the
following conditions hold:
\begin{itemize}
\item[\rm (i)] \eqref{e:ffgg} holds for the L\'evy measure $\nu$, and \begin{equation}\label{th10}\inf_{x,y\in\R^d:|x-y|\le \kappa_0}\mu_\Psi(\R^d)>0\end{equation} and
$$A_1:=\sup_{x,y\in \R^d}\left(\Lambda \mu_{\Psi}(\R^d)(|x-y|\wedge
\kappa_0)+(1+\Lambda^2/2)\int_{\R^d}|z|\,(\mu_{\Psi}+\mu_{\Psi^{-1}})(dz)\right)<\infty$$
for some $\kappa_0>0$;
\item[\rm (ii)] The diffusion coefficient $\sigma(x)$ and the drift term $b(x)$ satisfy
\begin{equation}\label{th111}\begin{split}
  \frac{\langle b(x)-b(y), x-y\rangle}{|x-y|}+\|\sigma(x)-\sigma(y)\|_{\mathrm{H.S.}}&(A_1+A_2)\\
  &\le \begin{cases}\Phi_1(|x-y|), &|x-y|< l_0,\\
  -K_2|x-y|,& |x-y|>l_0\end{cases}\end{split}\end{equation}
 for some constants $ K_2>0,\,l_0\ge 0$, and a nonnegative concave function $\Phi_1\in C([0,2l_0])\cap C^2((0,2l_0])$ such that $\Phi_1(0)=0$ and $\Phi''_1$ is nondecreasing, where $$A_2=\int_{\{|z|>1\}}|z|\,\nu(dz)+\frac{L_\sigma}{2}\int_{\{|z|\le
 1\}}|z|^2\,\nu(dz);$$
\item[\rm (iii)]  There is a nondecreasing and concave function $\sigma\in C([0,2l_0])\cap C^2((0,2l_0])$ such that for some $\kappa\in (0,\kappa_0]$, one has
  \begin{equation*}\label{thtpw.1}
  \sigma(r)\leq \frac1{2r} J(\kappa\wedge r) (\kappa\wedge r)^2, \quad r\in (0, 2l_0];
  \end{equation*}
and the integrals $g_1(r)=\int_{0}^r \frac{1}{\sigma(s)}\,ds$ and
$g_2(r)=\int_0^r \frac{\Phi_1(s)}{s\sigma(s)}\, ds$ are well defined
for all $r\in [0,2l_0]$, where $J(r)$ is defined by \eqref{e:llev}.
\end{itemize}
Set $c_2=(2K_2)\wedge g_1(2l_0)^{-1}$ and $c_1=e^{-c_2 g(2l_0)}$, where the function $g$ is defined by
  $$g(r)=g_1(r)+\frac2{c_2} g_2(r),\quad r\in (0,2l_0].$$ Then for any $x,y\in\R^d$ and $t>0$,
$$  W_{\psi}(\delta_xP_t, \delta_yP_t)\le e^{-\lambda t} \psi(|x-y|),$$
and so
  $$  W_1(\delta_xP_t, \delta_yP_t)\le C e^{-\lambda t} |x-y|,$$
where
  \begin{equation*}\label{pwcon}\begin{split} \psi(r)&= \begin{cases}
  c_1 r+ \int_0^r e^{-c_2 g(s)}\, ds ,& r\in [0, 2l_0],\\
 \psi(2l_0)+ \psi'(2l_0) (r-2l_0), & r\in(2l_0,\infty),\end{cases}\\
      \lambda&=\frac{c_2}{1+e^{c_2 g(2l_0)}}=\frac{(2K_2)\wedge g_1(2l_0)^{-1}}{1+ \exp\big\{ g(2l_0) \big[ (2K_2)\wedge g_1(2l_0)^{-1} \big] \big\}},
  \end{split}  \end{equation*} and $$C=\frac{1+c_1}{2c_1}=\frac1{2}\Big(1+\exp\big\{g(2l_0) \big[ (2K_2)\wedge g_1(2l_0)^{-1} \big] \big\}\Big).$$

\end{theorem}

\begin{theorem}\label{thtvart}
Assume that all assumptions but {\rm (ii)} and {\rm (iii)} in Theorem $\ref{thtpw}$
hold, and that the following two conditions $($replacing {\rm (ii)} and {\rm (iii)} respectively$)$ are satisfied
\begin{itemize}
\item[\rm (ii')] The diffusion coefficient $\sigma(x)$ and the drift term $b(x)$ satisfy
\begin{equation}\label{th111}\begin{split}
  \frac{\langle b(x)-b(y), x-y\rangle}{|x-y|}+\|\sigma(x)-\sigma(y)\|_{\mathrm{H.S.}}&(A_1+A_2) \\
  &\le \begin{cases} K_1, &|x-y|<l_0,\\
  -K_2|x-y|,&|x-y|\ge l_0\end{cases}\end{split}\end{equation}
 for some constants $ K_1\ge0,\, K_2>0$ and $l_0\ge 0$, where $$A_2=\int_{\{|z|>1\}}|z|\,\nu(dz)+\frac{L_\sigma}{2}\int_{\{|z|\le
 1\}}|z|^2\,\nu(dz).$$
 \item[\rm (iii')]  There is a nondecreasing and concave function $\sigma\in C([0,2l_0])\cap C^2((0,2l_0])$ such that for some $\kappa\in (0,\kappa_0\wedge l_0]$, one has
  \begin{equation*}\label{thtpw.1}
  \sigma(r)\leq \frac1{2r} J(\kappa\wedge r) (\kappa\wedge r)^2, \quad r\in (0, 2l_0];
  \end{equation*}
and the integral $g(r)=\int_{0}^r \frac{1}{\sigma(s)}\,ds$ is well defined for all $r\in [0,2l_0]$.

\end{itemize}
Then there exist constants $\lambda, c>0$ such that for any
$x,y\in\R^d$ and $t>0$,
  $$  \|\delta_xP_t- \delta_yP_t\|_{{\rm Var}}\le c e^{-\lambda t}(1+|x-y|).$$
\end{theorem}

\begin{proof}[Sketch of the proofs of Theorems $\ref{thtpw}$ and $\ref{thtvart}$] As we mentioned before, since \eqref{e:ffgg} holds, one can take $R=\infty$ in the estimate \eqref{p:est1}. Under assumptions on the diffusion
coefficient $\sigma(x)$ and condition {\rm (i)}, \eqref{p:est1} is
further reduced into \begin{align*}\widetilde{L}f(|x-y|)\le&
\Theta_0(f)(x,y)\\
&+ f'(|x-y|)\left( \frac{\langle b(x)-b(y),
x-y\rangle}{|x-y|}+\|\sigma(x)-\sigma(y)\|_{\mathrm{H.S.}}(A_1+A_2)\right).\end{align*}
Then, one can use the estimate above and follow arguments of
\cite[Theorem 4.2 and 4.4]{Lwcw} to prove required conclusions.
\end{proof}

\section{Proofs and examples}

Recall from \eqref{e:psi} that, for any $\kappa>0$ and $x,y\in \R^d$,
$$\Psi(z)=\Psi_{\kappa,x,y}(z)=\sigma(y)^{-1}\big(\sigma(x)z+(x-y)_\kappa\big),$$ where $(x-y)_\kappa=\big(1\wedge
\frac{\kappa}{|x-y|}\big)(x-y)$. Hence,
$$\Psi^{-1}(z)=\sigma(x)^{-1}\big(\sigma(y)z-(x-y)_\kappa\big).$$

\subsection{Estimates related to L\'evy measures}

To prove Theorems \ref{T:thm1}
and \ref{T:222}, we need both lower bound and upper bound for $$J(r)=\inf_{x,y\in\R^d:|x-y|\le r}\mu_\Psi(\R^d)=\inf_{x,y\in\R^d:|x-y|\le r}\mu_{\Psi^{-1}}(\R^d),$$ where $\mu_\Psi(\R^d)=(\nu\wedge (\nu\Psi))(\R^d)=(\nu\wedge (\nu\Psi^{-1}))(\R^d)=\mu_{\Psi^{-1}}(\R^d)$.
\begin{proposition}\label{L:low1}Suppose that there are $0<\varepsilon_1,\varepsilon_2\le 1$  and $c_0>0$ such that
$$\nu(dz)\ge \I_{\{-\varepsilon_1<z_1< \varepsilon_2\}}
\frac{c_0}{|z|^{d+\alpha}}\,dz.$$ Then, there exists a constant $c_1>0$ such that for all $0<r<\kappa$ small enough,
$$J(r)\ge c_1r^{-\alpha}.$$ \end{proposition}
\begin{proof}
Let $\nu_0(dz)=q(z)\,dz,$ where $$q(z)=\I_{\{-\varepsilon_1<z_1<
\varepsilon_2\}} \frac{c_0}{|z|^{d+\alpha}}.$$ Then,
$(\nu_0\Psi^{-1})(dz)= q_{\Psi}(z)\,dz,$ where
$$q_\Psi(z)=\I_{\{-\varepsilon_1<\Psi^{-1}(z)_1< \varepsilon_2\}}
\frac{c_0|C_\Psi(x,y)|}{\,\,|\Psi^{-1}(z)|^{d+\alpha}}$$ and $C_\Psi(x,y)$ is the
determinant of $\sigma(x)^{-1}\sigma(y)$, i.e., the Jacobian matrix
corresponding to the transformation $\Psi^{-1}(z)\mapsto z$. Hence,
  $$q(z)\wedge q_\Psi(z)\ge c_0^*\I_{\{-\varepsilon_1<z_1< \varepsilon_2, -\varepsilon_1<\Psi^{-1}(z)_1< \varepsilon_2\}} \bigg(\frac{1}{|z|^{d+\alpha}} \wedge
  \frac{1}{|\Psi^{-1}(z)|^{d+\alpha}}\bigg),$$ where $c^*_0:=c_0(1\wedge\, \inf_{x,y\in\R^d}|C_\Psi(x,y)|)>0$.

In the following,
we consider $x,y\in\R^d$ such that $|x-y|\leq \kappa\wedge
(\varepsilon_2/(4\Lambda^5))\wedge(\varepsilon_1/(4\Lambda))$, where
$\kappa>0$ is small enough such that
\begin{itemize}
\item[(i)] for
all $|z|\le \varepsilon_1+\varepsilon_2$ with $-\varepsilon_1/2<z_1<\varepsilon_2/2$, it holds
$-\varepsilon_1<(\sigma(x)^{-1}\sigma(y)z)_1<\varepsilon_2$;
\item[(ii)] for all $|z|\le \Lambda^2( \varepsilon_1+\varepsilon_2)$ with
$-\varepsilon_1/2<(\sigma(x)^{-1}\sigma(y)z)_1<3\varepsilon_2/4$, it holds true that $-\varepsilon_1<z_1<\varepsilon_2.$\end{itemize} These two properties above are ensured by the
boundedness and the globally  Lipschitz continuity of $\sigma(x)$.

If $(\sigma(x)^{-1}(x-y))_1\le 0$, then
  \begin{align*} q(z)\wedge q_\Psi(z)&\geq c^*_0\I_{\{-\varepsilon_1<z_1<\varepsilon_2, -\varepsilon_1<\Psi^{-1}(z)_1< \varepsilon_2\}} \bigg(\frac{1}{|z|^{d+\alpha}} \wedge
  \frac{1}{|\Psi^{-1}(z)|^{d+\alpha}}\bigg)\\
  &\ge c_0^*\I_{\{-\varepsilon_1/2<z_1< \varepsilon_2/2,  (\sigma(x)^{-1}\sigma(y)z)_1<
 \varepsilon_2+(\sigma(x)^{-1}(x-y))_1,|\sigma(x)^{-1}(x-y)|<|\sigma(x)^{-1}\sigma(y)z|\}}\\
 &\quad \times\bigg(\frac{1}{|z|^{d+\alpha}} \wedge   \frac{1}{(2|\sigma(x)^{-1}\sigma(y)z|)^{d+\alpha}}\bigg)\\
  &\geq \frac{c_0^*}{(2\Lambda^2)^{d+\alpha}}\I_{\{0<z_1< \varepsilon_2/2,
  |\sigma(x)^{-1}(x-y)|<|\sigma(x)^{-1}\sigma(y)z|<\varepsilon_2-|\sigma(x)^{-1}(x-y)|\}}\frac{1}{|z|^{d+\alpha}}\\
  &\ge  \frac{c_0^*}{(2\Lambda^2)^{d+\alpha}}\I_{\{z_1>0, \Lambda^3|x-y|<|z|<\varepsilon_2/(2\Lambda^2)\}}\frac{1}{|z|^{d+\alpha}}.
  \end{align*}
Thus, denoting by $S^{d-1}_+ =\{\theta\in \R^d: |\theta|=1 \mbox{
and } \theta_1>0\}$ the half sphere and $\sigma( d\theta)$ the
spherical measure, we have
  \begin{align*}
  \int_{\R^d} q(z)\wedge q_\Psi(z)\, dz &\geq \frac{c_0^*}{(2\Lambda^2)^{d+\alpha}} \int_{\{z_1>0,\Lambda^3|x-y|< |z|<\varepsilon_2/(2\Lambda^2)\}} \frac{1}{|z|^{d+\alpha}}\, dz\\
  &= \frac{c_0^*}{(2\Lambda^2)^{d+\alpha}} \int_{\Lambda^3|x-y|}^{\varepsilon_2/(2\Lambda^2)} r^{d-1}\, dr \int_{S^{d-1}_+}\frac{\sigma(d\theta)}{|r\theta|^{d+\alpha}}\\
  &= \frac{c_0^*\omega_d}{2(2\Lambda^2)^{d+\alpha}}\bigg( \frac1{(\Lambda^3|x-y|)^\alpha}- \frac1{(\varepsilon_2/(2\Lambda^2))^\alpha}\bigg),
  \end{align*}
where $\omega_d=\sigma(S_+^{d-1})$ is the area of the sphere. Noticing
that $|x-y|\leq \varepsilon_2/(4\Lambda^5)$, we further get that
$$
  \int_{\R^d} q(z)\wedge q_\Psi(z)\, dz\geq \frac{c_0^*\omega_d}{2^{1+d+\alpha}\Lambda^{2d+5\alpha}} \bigg(1-\frac{1}{2^\alpha}\bigg)\frac1{|x-y|^\alpha} .
$$

Next, we follow the argument above to consider the case that $(\sigma(x)^{-1}(x-y))_1> 0$. Note that
$$z=\sigma(y)^{-1}\sigma(x)(\Psi^{-1}(z)+\sigma(x)^{-1}(x-y)_\kappa).$$ In this case, it holds
 \begin{align*}&q(z)\wedge q_\Psi(z)\\
  &\geq\frac{ c^*_0}{\Lambda^{2(d+\alpha)}}\I_{\{-\varepsilon_1<z_1<\varepsilon_2, -\varepsilon_1<\Psi^{-1}(z)_1< \varepsilon_2\}}
  \frac{1}{(|\Psi^{-1}(z)|+|\sigma(x)^{-1}(x-y)|)^{d+\alpha}} \\
  &\ge \frac{ c^*_0}{\Lambda^{2(d+\alpha)}}\I_{\{-\varepsilon_1/2<\Psi^{-1}(z)_1<\varepsilon_2/2,|\sigma(x)^{-1}(x-y)|<|\Psi^{-1}(z)|\}}\frac{1}{(2|\Psi^{-1}(z)|)^{d+\alpha}}\\
     &\ge \frac{c_0^*}{(2\Lambda^2)^{d+\alpha}}\I_{\{0<\Psi^{-1}(z)_1<\varepsilon_2/2,|\sigma(x)^{-1}(x-y)|<|\Psi^{-1}(z)|\}} \frac{1}{|\Psi^{-1}(z)|^{d+\alpha}}\\
     &\ge   \frac{c_0^*}{(2\Lambda^2)^{d+\alpha}}\I_{\{\Psi^{-1}(z)_1>0,|\sigma(x)^{-1}(x-y)|<|\Psi^{-1}(z)|<\varepsilon_2/2\}} \frac{1}{|\Psi^{-1}(z)|^{d+\alpha}}\\
     &\ge   \frac{c_0^*}{(2\Lambda^2)^{d+\alpha}}\I_{\{\Psi^{-1}(z)_1>0,\Lambda|x-y|<|\Psi^{-1}(z)|<\varepsilon_2/2\}} \frac{1}{|\Psi^{-1}(z)|^{d+\alpha}}.
  \end{align*}
Hence, we arrive at
  \begin{align*}
  \int_{\R^d} q(z)\wedge q_\Psi(z)\, dz
    &\geq \frac{c_1^{*}}{(2\Lambda^2)^{d+\alpha}} \int_{\{z_1>0,\Lambda|x-y|\leq |z|\leq \varepsilon_2/2\}} \frac{1}{|z|^{d+\alpha}}\, dz\\
  &= \frac{c_1^{*}}{(2\Lambda^2)^{d+\alpha}} \int_{\Lambda|x-y|}^{\varepsilon_2/2} r^{d-1}\, dr \int_{S^{d-1}_+}\frac{\sigma(d\theta)}{|r\theta|^{d+\alpha}}\\
  &= \frac{c_1^{*}\omega_d}{2(2\Lambda^2)^{d+\alpha}}\bigg( \frac1{(\Lambda|x-y|)^\alpha}- \frac1{(\varepsilon_1/2)^\alpha}\bigg),
  \end{align*}
where $c^{*}_1:=c_0(1\wedge\, \inf_{x,y\in\R^d}|C_\Psi(x,y)|)/(\sup_{x,y\in\R^d}|C_\Psi(x,y)|)>0$. Since $|x-y|\leq \varepsilon_1/(4\Lambda)$, we further get that
$$
  \int_{\R^d} q(z)\wedge q_\Psi(z)\, dz\geq \frac{c_1^*\omega_d}{2^{1+d+\alpha}\Lambda^{2d+3\alpha}} \bigg(1-\frac{1}{2^\alpha}\bigg)\frac1{|x-y|^\alpha} .
$$

 Therefore, we conclude that for all $0<r\le \kappa\wedge
(\varepsilon_2/(4\Lambda^5))\wedge(\varepsilon_1/(4\Lambda))$,
$$J(r)\ge \inf_{x,y\in \R^d: |x-y|\le r}  \int_{\R^d} q(z)\wedge q_\Psi(z)\, dz\ge c_2 r^{-\alpha},$$ which finishes the proof.
\end{proof}

A close inspection of the proof above shows the following statement. %, which is reduced into \cite[Example 1.2]{Lwcw} for SDEs with additive L\'evy noises.

 \begin{corollary}\label{C:up1} \begin{itemize}\item[(1)]  If there are $0<\varepsilon\le 1$  and $c_0>0$ such that
$$\nu(dz)\ge \I_{\{|z|< \varepsilon\}}
\frac{c_0}{|z|^{d+\alpha}}\,dz,$$ then there exists a constant $c_1>0$ such that for all $0<r<\kappa$ small enough,
$$J(r)\ge c_1r^{-\alpha}.$$

\item[ (2)] Suppose that $\sigma(x)=(\sigma_{i,j}(x))_{d\times d}$ is diagonal, i.e.,\, $\sigma_{i,j}(x)=0$ for all $x\in \R^d$ and $1\le i\neq j\le d$. If there are $0<\varepsilon\le 1$  and $c_0>0$ such that
$$\nu(dz)\ge \I_{\{0<z_1< \varepsilon\}}
\frac{c_0}{|z|^{d+\alpha}}\,dz,$$ then there exists a constant $c_1>0$ such that for all $0<r<\kappa$ small enough,
$$J(r)\ge c_1r^{-\alpha}.$$ \end{itemize}\end{corollary}
\begin{proof}One can easily obtain (1) from the proof of Proposition \ref{L:low1}. For (2), we note that, when $\sigma(x)$ is diagonal, $z_1>0$ if and only if $(\sigma(x)^{-1}\sigma(y)z)_1>0$ for all $x,y\in \R^d$. Then, we also can follow the argument of Proposition \ref{L:low1} to get the desired assertion. \end{proof}

Next, we consider some upper bounds related to $\mu_\Psi$.

\begin{proposition}\label{L:up}\begin{itemize}
\item[(1)] Let $$\nu(dz)\le \frac{c_0}{|z|^{d+\alpha}}\,dz$$ for some $c_0>0$.  Then,
there exists a constant $c_1>0$ $($independent of $\kappa$$)$ such
that
$$\mu_\Psi(\R^d)\le c_1(|x-y|\wedge\kappa)^{-\alpha}.$$
\item[(2)] Let $$\nu(dz)\le \frac{c_0}{|z|^{d+\alpha}}\I_{\{|z_1|\le \eta\}}\,dz$$ for some $\eta, c_0>0$.  Then,
there exists a constant $c_2>0$ $($independent of $\kappa$ and
$\eta$$)$ such that
 $$\int_{\R^d}|z|\,(\mu_\Psi+\mu_{\Psi^{-1}})(dz)\le c_2 \begin{cases}
  (|x-y|\wedge \kappa)^{1-\alpha}+\eta^{1-\alpha},& \alpha\in(0,1),\\
   \log\left(\frac{\eta}{|x-y|\wedge \kappa}\right),& \alpha=1,\\
 (|x-y|\wedge \kappa)^{1-\alpha}, &
 \alpha\in(1,2).
 \end{cases}$$ In particular, the estimate above holds if $$\nu(dz)\le \frac{c_0}{|z|^{d+\alpha}}\I_{\{|z|\le \eta\}}\,dz$$ for some $\eta, c_0>0$. \end{itemize}
\end{proposition}
\begin{proof} (1)
Let $\nu_0(dz)=q(z)\,dz$, where $$q(z)=
\frac{c_0}{|z|^{d+\alpha}}.$$   Note that if $|z|\le
(2\Lambda^{3})^{-1}(|x-y|\wedge
  \kappa)$, then \begin{align*}|\Psi(z)|=&|\sigma(y)^{-1}\sigma(x)z+\sigma(y)^{-1}(x-y)_\kappa|\\
  \ge & |\sigma(y)^{-1}(x-y)_\kappa|-|\sigma(y)^{-1}\sigma(x)z| \ge (2\Lambda) ^{-1}(|x-y|\wedge
  \kappa).\end{align*} Thus,
  \begin{align*}(\nu_0\wedge (\nu_0\Psi))(\R^d)&\le \int_{\{|z|\ge
 (2\Lambda^{3})^{-1}(|x-y|\wedge \kappa)\}} \,\nu_0(dz)+\int_{\{|z|\le (2\Lambda^{3})^{-1}(|x-y|\wedge
  \kappa)\}}\,(\nu_0\Psi)(dz)\\
  &\le \int_{\{|z|\ge
 (2\Lambda^{3})^{-1}(|x-y|\wedge \kappa)\}}  \frac{c_0}{|z|^{d+\alpha}}\,dz+\int_{\{\Psi(z): |z|\le(2\Lambda^{3})^{-1}(|x-y|\wedge
  \kappa)\}}\,\nu_0(dz)\\
  &\le \int_{\{|z|\ge
 (2\Lambda^{3})^{-1}(|x-y|\wedge \kappa)\}}  \frac{c_0}{|z|^{d+\alpha}}\,dz+\int_{\{|z|\ge(2\Lambda)^{-1}(|x-y|\wedge
  \kappa)\}}\frac{c_0}{|z|^{d+\alpha}}\,dz\\
  &\le c_1(|x-y|\wedge \kappa)^{-\alpha}.\end{align*} This along with the fact that $$\mu_\Psi=\nu\wedge (\nu\Psi)\le \nu_0\wedge (\nu_0\Psi) $$ proves
  the assertion (1).

(2) Let $\nu_0(dz)=q(z)\,dz$, where $$q(z)=
\frac{c_0}{|z|^{d+\alpha}}\I_{\{|z_1|\le \eta\}}.$$ Then,
\begin{align*}&\int_{\R^d} |z|\,(\nu_0\wedge (\nu_0\Psi))(dz)\\
&\le \int_{\{ |z_1|\le \eta, (2\Lambda^{3})^{-1}(|x-y|\wedge
\kappa)\le |z|\}}|z| \,\nu_0(dz)+\int_{\{|z|\le
(2\Lambda^{3})^{-1}(|x-y|\wedge
  \kappa)\}}\,|z|(\nu_0\Psi)(dz)\\
  &\le \int_{\{ |z_1|\le \eta,
(2\Lambda^{3})^{-1}(|x-y|\wedge \kappa)\le |z|\}}|z|
\,\nu_0(dz)\\
 &\quad+(2\Lambda^{3})^{-1}(|x-y|\wedge
  \kappa)\int_{\{|z|\le (2\Lambda^{3})^{-1}(|x-y|\wedge
  \kappa)\}}\,(\nu_0\Psi)(dz)\\
  &=:I_1+I_2.\end{align*}
Following the argument in (1), we know that there is a constant
$c_3>0$ such that for all $x,y\in \R^d$,
$$I_2\le c_3 (|x-y|\wedge \kappa)^{1-\alpha}.$$
For $I_1$, we denote by $\tilde x_{d-1}=(x_2,\cdots,x_d)$ for any
$x=(x_1,x_2,\cdots,x_d)$. Then,
\begin{align*} I_1&\le c_4\int_{-\eta}^{\eta}\,dz_1\int_{\{|z_1|+|\tilde z_{d-1}|\ge
c_5(|x-y|\wedge \kappa)\}} \frac{1}{(|z_1|+|\tilde
z_{d-1}|)^{d+\alpha-1}}\,d\tilde z_{d-1}\\
&\le  c_6\int_0^\eta\,ds\int_{\{s+r>c_5(|x-y|\wedge \kappa)\}}
\frac{r^{d-2}}{(s+r)^{d+\alpha-1}}\,dr\\
&= c_6\int_{c_5(|x-y|\wedge \kappa)}^\eta\,ds\int_0^\infty
\frac{r^{d-2}}{(s+r)^{d+\alpha-1}}\,dr\\
&\quad+c_6\int^{c_5(|x-y|\wedge
\kappa)}_0\,ds\int_{c_5(|x-y|\wedge \kappa)-s}^\infty
\frac{r^{d-2}}{(s+r)^{d+\alpha-1}}\,dr\\
&\le c_7\int_{c_5(|x-y|\wedge \kappa)}^\eta s^{-\alpha}\,ds+
c_7(|x-y|\wedge \kappa)^{1-\alpha}.
\end{align*}
From the estimate above, we can get $$I_1\le c_8\begin{cases}
(|x-y|\wedge \kappa)^{1-\alpha}+\eta^{1-\alpha},& \alpha\in(0,1),\\
   \log\left(\frac{\eta}{|x-y|\wedge \kappa}\right),& \alpha=1,\\
 (|x-y|\wedge \kappa)^{1-\alpha}, &
 \alpha\in(1,2).\end{cases}$$
All the estimates hold if we reply $\Psi$ with $\Psi^{-1}$.
Therefore, according to all inequalities above and the facts that
$$\mu_\Psi=\nu\wedge (\nu\Psi)\le \nu_0\wedge (\nu_0\Psi)$$
and $$\frac{c_0}{|z|^{d+\alpha}}\I_{\{|z|\le \eta\}}\le \frac{c_0}{|z|^{d+\alpha}}\I_{\{|z_1|\le \eta\}},$$  we can
prove the second desired assertion.
\end{proof}

\subsection{Proofs of Theorems \ref{T:thm1}, \ref{T:222} and Corollary \ref{C:c3}}

We now are in a position to present proofs of Theorems \ref{T:thm1}
and \ref{T:222}.

\begin{proof}[Proof of Theorem $\ref{T:thm1}$]
Define $\mu_\Psi=\nu_0\wedge (\nu_0\Psi)$, where
$$\nu_0(dz)=\I_{\{|z|<\eta\}}\frac{c_0}{|z|^{d+\alpha}}\,dz$$ for general case, and $$\nu_0(dz)=\I_{\{0<z_1<\eta\}}\frac{c_0}{|z|^{d+\alpha}}\,dz$$ when $\sigma(x)$ is diagonal.
Without loss of generality, in the following we can assume that
$\eta>0$ is small enough.

(1) Suppose that $\alpha\in (1,2)$. Then, by Corollary \ref{C:up1} and Proposition \ref{L:up}, there exist constants $c_1,c_2>0$ such that for all $r>0$ small enough,
$J(r)\ge c_1r^{-\alpha}$ and $K(r)\le c_2 r^{1-\alpha}$. Letting $\theta>0$, we take $\psi\in C_b([0,\infty))\cap C^3((0,\infty))$ such that $$\psi(r)=r(1-\log^{-\theta}(1/r))$$ for all $r>0$ small enough, which can be extended on $(0,2]$ such that $\psi'\ge0$, $\psi''\le0$ and $\psi'''\ge0$ on $(0,2]$. With this function $\psi$ and estimates for $J(r)$ and $K(r)$ above, we have for any $c_1^*,c_2^*>0$ and for $r>0$ small enough,
\begin{align*}&J(r)r^2\psi''(2r)+c^*_1K(r)\psi'(r)r+c^*_2\psi'(r)r^\beta\\
&\le c_3r^{1-\alpha}\big(-\log^{-(1+\theta)}(1/r)+c_4r+c_5r^{\beta-1+\alpha}\big)\\
&\le -c_6r^{1-\alpha} \log^{-(1+\theta)}(1/r),\end{align*} where we
used the fact that $\beta>1-\alpha$ in the last inequality.  Hence,
the first required assertion follows from Theorem \ref{T:re}.

(2) When $\alpha\in(0,1]$, we take $\psi\in C_b([0,\infty))\cap C^3((0,\infty))$ such that $\psi(r)=r^\theta$ on $[0,2]$ with $\theta\in(0,\alpha)$. Then, following the argument above and using Theorem \ref{T:re}, we can obtain the second assertion.
 \end{proof}

 \begin{proof}[Proof of Theorem $\ref{T:222}$] Note that $\beta\in (0,1]$, and $\sigma(x)$ is bounded and Lipschitz continuous. Under \eqref{e:tss1}, we can take $K^*_1\ge K_1$ and $l^*_0\ge l_0$ large enough such that
 \begin{equation}\label{e:tss}\begin{split}
  \frac{\langle b(x)-b(y), x-y\rangle}{|x-y|}&+A\|\sigma(x)-\sigma(y)\|_{\mathrm{H.S.}}  \le \begin{cases} K^*_1|x-y|^\beta,&\quad|x-y|<l^*_0,\\
  K_2|x-y|,&\quad|x-y|\ge l^*_0\end{cases}\end{split}\end{equation}
 holds for all $x,y\in \R^d$, where
 $$A:=\int_{\{|z|>1\}}|z|\,\nu(dz)+\frac{L_\sigma}{2}\int_{\{|z|\le
 1\}}|z|^2\,\nu(dz).$$

 Similar to the proof of Theorem \ref{T:thm1}, we define $\mu_\Psi=\nu_0\wedge (\nu_0\Psi)$, where
$$\nu_0(dz)=\I_{\{|z|<\eta\}}\frac{c_0}{|z|^{d+\alpha}}\,dz$$ for general case, and $$\nu_0(dz)=\I_{\{0<z_1<\eta\}}\frac{c_0}{|z|^{d+\alpha}}\,dz$$ when $\sigma(x)$ is diagonal.
   We first consider the case that $\alpha\in(0,1)$.
 According to
Corollary \ref{C:up1} and Proposition \ref{L:up}, by choosing
$\eta$ and $\kappa$ small enough, we
have that for $r>0$ small enough,
 $$J(r)\ge c_1r^{-\alpha}$$ and
$$ A_1:=\sup_{x,y\in \R^d}\left(\Lambda \mu_{\Psi}(\R^d)(|x-y|\wedge
\kappa_0)+(1+\Lambda^2/2)\int_{\R^d}|z|\,(\mu_{\Psi}+\mu_{\Psi^{-1}})(dz)\right)\le \frac{K_2}{2L_\sigma}.$$ Hence, it follows from \eqref{e:tss} that for all $x,y\in \R^d$,
\begin{align*}
  &\frac{\langle b(x)-b(y), x-y\rangle}{|x-y|}+(A+A_1)\|\sigma(x)-\sigma(y)\|_{\mathrm{H.S.}}\\
  &\le K^*_1|x-y|^\beta+K_2|x-y|/2-\big[K^*_1|x-y|^\beta+K_2|x-y|\big]\I_{\{|x-y|\ge l^*_0\}}.\end{align*} So, assumption (ii) in Theorem \ref{thtpw} holds with $\Phi_1(r)=K^*_1r^\beta+K_2r/2$. It is clear that, in this setting assumptions (i) and (iii) in Theorem \ref{thtpw} are satisfied too. In particular, (iii) holds with $\sigma(r)=c_2r^{1-\alpha}$ for some $c_2>0$, thanks to the assumption that $\beta>1-\alpha$. Therefore, the required assertion with respect to the Wasserstein distance between $\delta_x P_t$ and $\delta_y P_t$ follows from Theorem \ref{thtpw}. Similarly, the assertion about the total variation between $\delta_x P_t$ and $\delta_y P_t$ is a consequence of Theorem \ref{thtvart}.

When $\alpha\in[1,2)$, we choose $\varepsilon\in (1-\beta,1)$ and define $\mu_\Psi=\nu_0\wedge (\nu_0\Psi)$,
 where $$\nu_0(dz)=\I_{\{|z|<\eta\}}\frac{c_0}{|z|^{d+\varepsilon}}\,dz$$  for general case, and $$\nu_0(dz)=\I_{\{0<z_1<\eta\}}\frac{c_0}{|z|^{d+\varepsilon}}\,dz$$ when $\sigma(x)$ is diagonal. Then, following the argument above, we know that (iii) in Theorem \ref{thtpw} holds with $\sigma(r)=c_3r^{1-\varepsilon}$ for some $c_3>0$, and so we can obtain the desired conclusions.
 \end{proof}

Finally, we give the
\begin{proof}[Proof of Corollary $\ref{C:c3}$] It follows from \eqref{e:tss1} that for all $x\in\R^d$ with $|x|$ large enough,
  $$\frac{\langle b(x),x\rangle}{|x|}\le -K_2|x|+\frac{\langle b(0),x\rangle }{|x|}\le -\frac{K_2}{2}|x|.$$
 Let $f\in C^3(\R^d)$ such that $f(x)=|x|$ for all $|x|\ge 1$. Then, by \eqref{e:ffgg} and the assumption that $\sigma(x)$ is bounded, for any $x\in \R^d$,
  \begin{equation}\label{ly}\begin{split}
  Lf(x)&=\int\big(f(x+\sigma(x)z)-f(x)-\langle \nabla f(x), \sigma(x)z\rangle\I_{\{|z|\le 1\}}\big)\,\nu(dz)\\
  &\quad+\langle \nabla f(x),b(x)\rangle\\
  &\le \frac{\Lambda^2}{2}\|\nabla^2f\|_\infty \int_{\{|z|\le 1\}}|z|^2\,\mu(dz)+\Lambda\|\nabla f\|_\infty  \int_{\{|z|\ge1\}}|z|\,\nu(dz)+\langle \nabla f(x),b(x)\rangle\\
  &\le - c_2|x|+c_3\le -c_4f(x)+c_5,\end{split}
  \end{equation}
where $L$ is the generator of the process $X$, and $c_i$ $(i=2,\cdots,5)$ are positive constants. On the other hand, by Theorem \ref{T:222}, there exist constants $\lambda,c>0$ such that for any $x,y\in\R^d$ and $t>0$,
  $$W_1(\delta_x P_t,\delta_yP_t)\le c e^{-\lambda t} |x-y|,$$
which yields that (e.g.,\, see \cite[Theorem 5.10]{Chen1})
  $$\|P_tf\|_{{\rm Lip}}\le c e^{-\lambda t}\|f\|_{{\rm Lip}}$$
holds for any $t>0$ and any Lipschitz continuous function $f$, where $\|f\|_{{\rm Lip}}$ denotes the Lipschitz semi-norm with respect to the Euclidean norm $|\cdot|$. By the standard approximation, we know that the semigroup $(P_t)_{t\ge0}$ is Feller, i.e.,\, for every $t>0$, $P_t$ maps $C_b(\R^d)$ into $C_b(\R^d)$. This along with the Foster--Lyapunov type condition \eqref{ly} and \cite[Theorems 4.5]{MT} yields that the process $(X_t)_{t\ge0}$ has an invariant probability measure such that whose first moment is finite. Next, we claim that the process $X$ has a unique invariant probability measure. Indeed, let $\mu_1$ and $\mu_2$ be invariant probability measures of the process $X$ such that both of them have finite moment. Then, by Theorem \ref{T:222},
  \begin{align*}\|\mu_1-\mu_2\|_{{\rm Var}}&=\sup_{\|f\|_\infty\le 1}|\mu_1(f)-\mu_2(f)|\\
  &\le\sup_{\|f\|_\infty\le1}\iint|P_tf(x)-P_tf(y)|\,\mu_1(dx)\,\mu_2(dy)\\
 &\le \iint\|\delta_xP_t-\delta_yP_t\|_{{\rm Var}}\,\mu_1(dy)\,\mu_2(dx)\le ce^{-\lambda t}.\end{align*}
Letting $t\to\infty$, we find that $\mu_1=\mu_2$.

According to Theorem \ref{T:222}, we can get that for any probability measures $m_1$ and $m_2$ and any $t>0$,
$$W_1(m_1P_t, m_2P_t)\le c e^{-\lambda t} W_1(m_1,m_2),$$ e.g.,\ see \cite[Section 3]{M15}. Thus, for any $t>0$ and $x\in \R^d$,
$$W_1(\delta_xP_t, \mu P_t)\le c e^{-\lambda t} W_1(\delta_x,\mu)\le c_1(x)  e^{-\lambda t}.$$ Also by Theorem \ref{T:222},
\begin{align*}\|\delta_xP_t-\mu\|_{{\rm Var}} \le& \int  \|\delta_xP_t-\delta_yP_t\|_{{\rm Var}}\,\mu(dy)\le ce^{-\lambda t}\int (1+ |x-y|)\,\mu(dy)\\
 \le& c_2(x)e^{-\lambda t}.\end{align*}
The proof is complete.\end{proof}

\ \

\noindent{\bf Acknowledgements.} The research is supported by National
Natural Science Foundation of China (No.\ 11522106), the Fok Ying Tung
Education Foundation (No.\ 151002), %National Science Foundation of Fujian Province (No.\ 2015J01003),
and the Program for Probability and Statistics: Theory and Application (No. IRTL1704).

\end{document}